\documentclass[11pt]{article}
\usepackage{amssymb}
\usepackage{amsmath,amsthm}
\usepackage{amsfonts}
\usepackage{graphicx}
\usepackage{bm}

\def\ftoday{le \space\number\day \space\ifcase\month\or
  janvier\or f\'evrier\or mars\or avril\or mai\or juin\or
  juillet\or ao\^ut\or septembre\or octobre\or novembre\or d\'ecembre\fi
  \space\number\year}

\newtheorem{theorem}{Theorem}

\newtheorem{lemma}[theorem]{Lemma}

\newtheorem{remark}[theorem]{Remark}

\input amssym.def
\input amssym

\def\Ga{{\Gamma}}

\def\De{\Delta}

\def\ep{{\epsilon}}

\def\la{\lambda}
\def\La{\Lambda}

\newcommand{\card}{\operatorname{card}}

\newcommand{\cH}{\mathcal{H}}
\renewcommand{\cH}{\mathcal{H}}

\def\bk{{\bf k}}
\def\bK{{\bf K}}
\def\bm{{\bf m}}
\input{tcilatex}

\def\beq{\begin{equation}}
\def\eeq{\end{equation}}

\begin{document}

\title{Existence of quasipatterns solutions of the Swift-Hohenberg equation }
\author{Boele Braaksma\thanks{%
University of Groningen, Department of Mathematics P.O. Box 800, 9700 AV
Groningen, The Netherlands. E-mail : \texttt{braaksma@rug.nl}} \hspace{1mm},
G\'erard Iooss\thanks{%
Institut Universitaire de France-Laboratoire J.-A. Dieudonn\'e U.M.R. 7351,
Universit\'e de Nice - Sophia Antipolis, Parc Valrose 06108 Nice Cedex 02,
France. E-mail : \texttt{gerard.iooss@unice.fr} } and Laurent Stolovitch 
\thanks{%
CNRS-Laboratoire J.-A. Dieudonn\'e U.M.R. 7351, Universit\'e de Nice -
Sophia Antipolis, Parc Valrose 06108 Nice Cedex 02, France. E-mail : \texttt{%
stolo@unice.fr}. Ce travail a b\'en\'efici\'e d'une aide de l'Agence
Nationale de la Recherche portant la r\'ef\'erence ``ANR-10-BLAN 0102''}}
\maketitle
\date{}

\begin{abstract}
We consider the steady Swift - Hohenberg partial differential equation. It
is a one-parameter family of PDE on the plane, modeling for example Rayleigh
- B\'enard convection. For values of the parameter near its critical value,
we look for small solutions, quasiperiodic in all directions of the plane
and which are invariant under rotations of angle $\pi/q$, $q\geq 4$. We
solve an unusual small divisor problem, and prove the existence of solutions
for small parameter values.

\end{abstract}

\tableofcontents

\section{Introduction}

\label{sec:intro}

In the present paper we study the existence of a special kind of \textbf{%
stationnary solutions} (i.e. independent of $t$), bifurcating from $0$ (i.e. tending towards zero when the parameter $\lambda$ tends towards $0$),
called \textbf{quasipatterns} of the 2-dimensional Swift-Hohenberg PDE

\begin{equation}
\frac{\partial u}{\partial t}=\lambda u-(1+\Delta )^{2}u-u^{3}
\label{eq:shtime}
\end{equation}
where $u$ is the unknown real-valued function on some subset of $\mathbb{%
R^{+}\times 
\mathbb{R}
}^{2}$, $\Delta :=\left( \frac{\partial ^{2}}{\partial x_{1}^{2}}+\frac{%
\partial ^{2}}{\partial x_{2}^{2}}\right) $ and $\lambda $ is a parameter.
These are two-dimensional patterns that have no translation symmetry and are
quasiperiodic in any spatial direction.

Mathematical existence of quasipatterns is one of the outstanding problems
in pattern formation theory. To our knowledge, hereafter is the first proof
of existence of such quasipatterns of a PDE. Quasipatterns were discovered
in nonlinear pattern-forming systems in the Faraday wave experiment~\cite%
{Christiansen1992,Edwards1994}, in which a layer of fluid is subjected to
vertical oscillations. Since their discovery, they have also been found in
nonlinear optical systems, shaken convection and in liquid crystals (see
references in \cite{Arg-Io}) . In spite of the lack of translation symmetry
(in contrast to periodic patterns), the solutions are $\pi/q$-rotation invariant for some integer $q$ (most often observed, $2q$ is $8$, $10$
or $12$).

In many of these experiments, the domain is large compared to the size of
the pattern, and the boundaries appear to have little effect. Furthermore,
the pattern is usually formed in two directions ($x_1$ and $x_2$), while the
third direction ($z$) plays little role. Mathematical models of the
experiments are therefore often posed with two unbounded directions, and the
basic symmetry of the problem is $E(2)$, the Euclidean group of rotations,
translations and reflections of the $(x_{1},x_{2})$ plane.

The above model equation is the simplest pattern-forming PDE, and is
extremely successful for describing primary bifurcations (the first symmetry breaking) of hydrodynamical
instability problems such as the Rayleigh - B\'enard convection. Its
essential properties are that

i) the system is invariant under the group $E(2)$;

ii) the instability occurs for a certain critical value of the parameter
(here $\la=0$) for which critical modes are given by wave vectors sitting on
a circle of \emph{non zero radius} (here the unit circle);

iii) the linear part is selfadjoint and contains the main derivatives.

In contrast to periodic patterns, quasipatterns do not fit into any
spatially periodic domain and have Fourier expansions with wavevectors that
live on a \emph{quasilattice} (defined below). At the onset of pattern
formation, the critical modes have zero growth rate but there are other
modes on the quasilattice that have growth rates arbitrarily close to zero,
and techniques that are used for periodic patterns cannot be applied. These
small growth rates appear as \emph{small divisors}, as seen below, and
correspond at criticality ($\la 
=0) $ to the fact that for the linearized operator at the origin (here $%
-(1+\Delta )^{2}$), {\bf the 0 eigenvalue is not isolated in the spectrum,
appearing as part of the continuous spectrum}.

If a formal computation in powers of $\sqrt{\la}$ is performed in this case without regard to its
validity, this results in a possibly divergent power series in the
parameter, and this approach does not lead to the existence of quasipattern
solutions, but instead to approximate solutions up to an exponentially small
error ~\cite{iooss-rucklidge}.

In this work, \textbf{we prove the existence of quasipattern solutions of
the steady Swift-Hohenberg equations}. Our result rests on the article \cite%
{iooss-rucklidge} by G. Iooss and A.M. Rucklidge which settle the
mathematical foundation of the problem such as the formulation of suitable
functions spaces. We refer to the articles of Rucklidge \cite%
{rucklidge-rucklidge,rucklidge-siber} and Iooss-Rucklidge\cite%
{iooss-rucklidge} for physical motivation as well as for the bibliography. 
\begin{figure}[tbp]
\begin{center}
\includegraphics[width=0.9\hsize]{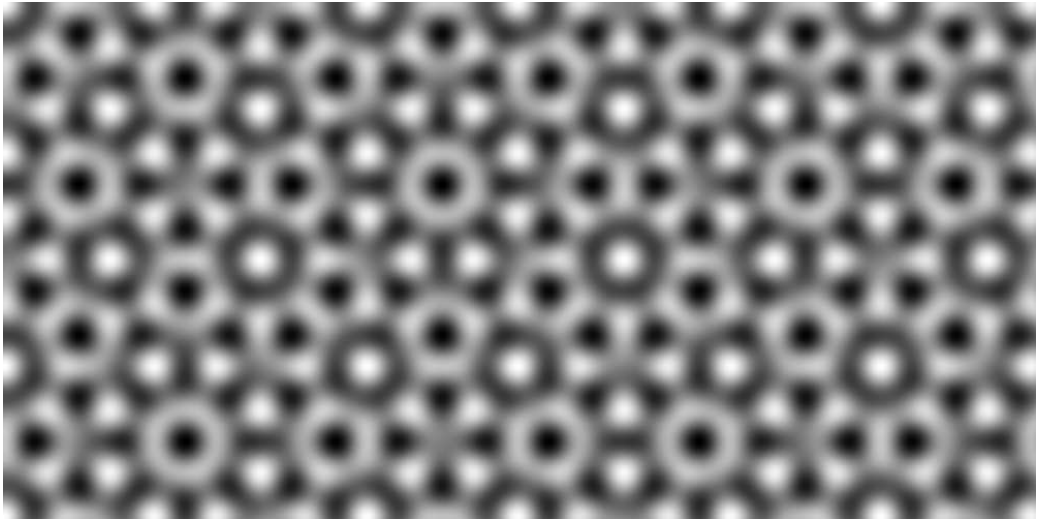}
\end{center}
\caption{Example $8$-fold quasipattern after \protect\cite{iooss-rucklidge}.
This is an approximate solution of the steady Swift--Hohenberg equation~(%
\protect\ref{steadySH}) with $\la=0.1$, computed by using Newton iteration
to find an equilibrium solution truncated to wavenumbers satisfying $|%
\mathbf{k}|\leq \protect\sqrt{5}$ and to the quasilattice $\Gamma _{27}$
obtained with $N_{\mathbf{k}}\leq 27$.}
\end{figure}

\subsection{Main result and sketch of the proof}

The problem is to find a special kind of solutions defined on $\mathbb{%
\mathbb{R}
}^{2}$ of the steady Swift-Hohenberg equation 
\begin{equation}
(1+\Delta )^{2}U-\lambda U+U^{3}=0.  \label{steadySH}
\end{equation}
The parameter $\lambda$ is supposed to be real and small in absolute value. The solutions we are interested in should tend towards zero as the parameter goes to zero.

We study equation (\ref{steadySH}) for $\lambda >0$. Namely, let $Q=2q$ be
an even integer and let $\mathbf{k}_{j}=\exp \frac{i\pi (j-1)}{q}$, $%
j=1,\ldots ,2q$ be the $2q$ unit vectors of the plane, identified with roots
of unity. Let $\Gamma $ be the set of linear combinations of vectors $%
\mathbf{k}_{j}$ with nonnegative integer coefficients. We look for the
existence of a (nonzero) $\pi /q$-rotation invariant solution of the form 
\begin{equation*}
U(\mathbf{x})=\sum_{\mathbf{k}\in \Gamma }u^{(\mathbf{k})}e^{i\mathbf{k}.%
\mathbf{x}}
\end{equation*}%
which belongs to a \textquotedblleft Sobolev\textquotedblright\ like space $%
\mathcal{H}_{s}$, $s\geq 0$~: 
\begin{equation*}
\Vert U\Vert _{s}^{2}:=\sum_{\mathbf{k\in \Gamma }}|u^{(\mathbf{k)}%
}|^{2}(1+N_{\mathbf{k}}^{2})^{s}<+\infty .
\end{equation*}%
The natural number $N_{\mathbf{k}}$ denotes the minimal length of the linear
combinations of the $\mathbf{k}_{j}$'s needed to reach $\mathbf{k}$.

We then show that such a solution exists indeed, for small positive
parameters $\lambda $:

\begin{theorem}
For any $q\geq 4$ and any $s>q/2$, there exists $\lambda _{0}>0,$ such that
the steady Swift-Hohenberg equation for $0<\lambda <\lambda _{0}$, admits a
quasipattern solution in $\mathcal{H}_{s}$, invariant under rotations of
angle $\pi /q.$ Moreover the asymptotic expansion of this solution is given
by the formal series given in \cite{iooss-rucklidge}.
\end{theorem}

One of the main difficulty is that the linearized operator at $W=0$, has an unbounded inverse. Indeed, it is easy to show that the eigenvalues of $(1+\De)^2$ in ${\cal H}_s$ are $(1-|\bk|^2)^2$ where $\bk\in\Ga$. These numbers accumulate in $0$. It creates a {\it small divisor} problem, such that if $\la>0$ nothing can be said {\it a priori} about $(1+\De)^2-\la\text{Id}$. We use the first terms of the asymptotic expansion of the solution and change the unknown as $U=U_{\ep}+\ep^4 W$ and $\la=\ep^2(\la_2+\la_4\ep^2)$ for some well chosen $U_{\ep}$ and positive $\la_2$. Let ${\cal L}_{_\ep}$ be the linear part at $W=0$ of the nonlinear equation so obtained. 
For $\epsilon=0$, the operator $\mathcal{L}_0 =(1+\Delta)^2$ is a positive selfadjoint operator in $\mathcal{H}_s$. It is bounded from $\mathcal{H}_{s+4}$ into $\mathcal{H}_s$, but it is not Fredholm, since its range is not closed. Its spectrum is an \emph{essential spectrum} filling the half line 
$[0,\infty)$. The set of eigenvalues is dense in the spectrum. The linear operator  $\cal L_{\ep}$ is the sum of $\mathcal{L}_0$ and a bounded operator (multiplication by a small function $O(\ep^2)$) selfadjoint in $\mathcal{H}_0$. If the range of $\mathcal{L}_{\ep}$ were closed, a usual way to estimate the inverse of the selfadjoint operator $\mathcal{L}_{\ep}$ in $\mathcal{H}_0$, would be to estimate the distance from $0$ to its numerical range (see \cite{kato-book}) (containing the spectrum). Such estimate as
\begin{equation*}
{ \langle\cal L_{\ep}U,U\rangle}_0 \geq c \ep^2||U||_{0}^{2}
\end{equation*}
for a certain constant $c>0$, cannot be proved here. So we need to study the linear operator in more details.

We show that there exists an orthogonal decomposition (depending on $\ep$) of the space ${\cal H}_s=E_0\oplus E_1\oplus E_2$, $s>q/2$, such that the solution of the equation ${\cal L}_{\ep}U=f$ in ${\cal H}_s$ can be computed and estimated from its $E_2$-component $U_2$. The latter is solution of a linear equation ${\cal L}^{(2)}_{\ep}U_2=\tilde f$. The main part (with respect to powers of $\ep$) of that operator is an operator $\Lambda _{\epsilon }$. When restricted to rotation invariant elements, it has a block-diagonal structure of fixed finite dimensional blocks. Then, it is possible to estimate all eigenvalues of these selfadjoint blocks. These eigenvalues have the form
$$
(|\bk|^2-1)^2+3\ep^2+O(\ep^4)
$$
for $\bk\in \Ga$. One of the main feature is that {\bf they do not accumulate at the origin}, and
\emph{despite of the small divisor problem} arising for $\epsilon =0$, we are able to give an upper bound in $\mathcal{H}_0$ of the inverse of $\Lambda _{\epsilon }$  of order $1/\lambda $ for nonzero $\lambda $.
Then we extend this estimate in $\mathcal{H}_{s}$ for $s>q/2$.
Finally, we use a variant of the implicit function theorem to conclude to the existence of quasipatterns solutions.

\begin{remark}
If the coefficient (3) of $\ep^2$ were negative, then the operator could have again {\it small divisors}. Then the proof of the existence would have been much more involved. At that point, solving the nonlinear problem in $U_2$ would have required the use of a version of Nash-Moser theorem such as those developed by J. Bourgain, W. Craig, M. Berti and al. (see for instance  \cite{bourgain,craig-panorama,berti-bolle-abstract, berti-book}). Their main feature is the use of the good separation property of the ``singular sites'' of the main linear operator. Indeed, we can show that the operator $(1+\De)^2$ does have this property.
\end{remark}
\begin{remark}
For $\lambda <0$ the solution $U=0$ is isolated in an open ball of radius $%
\sqrt{|\lambda |}$, as it is easily deduced from the following estimate 
\begin{equation*}
||[(1+\Delta )^{2}-\lambda ]^{-1}||\leq |\lambda |^{-1}
\end{equation*}%
which holds in $\mathcal{H}_s$.
\end{remark}
\section{Analysis of the main part of the differential}

\subsection{Setting}

In this section, we recall and improve some of the properties of the function spaces we use, as defined in \cite{iooss-rucklidge}.

Let $Q=2q$, $q\geq 4$ being an integer. Let us define the unit {\it wave vectors} (identifying $\mathbb{C}$ with $\mathbb{R}^2$)
\beq
\bk_j:=e^{i\pi\frac{j-1}{q}},\quad j=1,\ldots, 2q.
\eeq

We define the {\bf quasilattice} $\Ga\subset \Bbb R^2$ to be the set of points spanned by (nonnegative) integer linear combinations of the $k_j$'s : 
\beq\label{defkm}
\bk_m=\sum_{j=1}^{2q}m_j\bk_j,\quad \bm=(m_1,\ldots,m_{2q})\in \Bbb N^{2q}.
\eeq
We have $\bk_j=-\bk_{j+q}$. Hence, we can write 
$$
\bk_m=\sum_{j=1}^{q}m_j' \bk_j,
$$
where, $m_j':=m_j-m_{j+q}$ belongs to $\Bbb Z$. Thus, 
$$
|\bk_m|^2=\sum_{i,j=1}^{q}m_i'm_j'<\bk_i,\bk_j>.
$$
We then define, for any $\bm\in \Bbb N^{2q}$ and $\bk\in \Ga$, 
$$
|\bm|:=\sum_jm_j,\quad N_{\bk}:=\min\{|\bm| : \bk=\bk_m\}.
$$
We have 
\begin{lemma}\cite{iooss-rucklidge}[Lemma 4.1]\label{lem-iooss}
For any $\bk\in \Ga$, we have~:
\begin{itemize}
\item \beq\label{N1} N_{\bk+\bk'}\leq N_{\bk}+N_{\bk'},\quad N_{-\bk}=N_{\bk}\eeq
\item \beq\label{N2} |\bk|\leq N_{\bk}\eeq
\item \beq\label{N3} \card\{\bk\;|\; N_{\bk}=N\}\leq c_1(q)N^{q-1}\eeq
for some constant $c_1(q)$ depending only on $q$.
\end{itemize}
\end{lemma}

As in \cite{iooss-rucklidge}, we use function spaces defined as 
\begin{equation}
{\mathcal{H}}_{s}=\left\{ W=\sum_{\mathbf{k}\in \Gamma }W^{(\mathbf{k})}e^{i%
\mathbf{k}\cdot \mathbf{x}};\,||W||_{s}^{2}=\sum_{\mathbf{k}\in \Gamma
}(1+N_{\mathbf{k}}{}^{2})^{s}|W^{(\mathbf{k})}|^{2}<\infty \right\} ,
\label{defHs}
\end{equation}%
which are Hilbert spaces with the scalar product%
\begin{equation}
\langle W,V\rangle _{s}=\sum_{\mathbf{k}\in \Gamma }(1+N_{\mathbf{k}%
}{}^{2})^{s}W^{(\mathbf{k})}\overline{V}^{(\mathbf{k})}.
\label{scalarproduct}
\end{equation}%


\begin{lemma} \label{algebraa}
For $s>q/2$, for any $U\in \mathcal{H}_{s}$ and any $V\in 
\mathcal{H}_{0}$, we have%
\begin{equation*}
||UV||_{0}\leq c_{s}||U||_{s}||V||_{0}
\end{equation*}%
for a certain constant $c_{s}>0.$
\end{lemma}
\begin{proof}

Using Cauchy-Schwarz inequality, we have
\begin{eqnarray*}
\|UV\|_{0}^2 & \leq & \sum_{\bk\in \Ga}\left|\sum_{\bm\in \Ga}U^{(\bm)}V^{(\bk-\bm)}\right|^2\\
& \leq & \sum_{\bk\in \Ga}\left(\sum_{\bm\in \Ga}|U^{(\bm)}|^2(1+N_{\bm}^2)^{s}\right)
\left(\sum_{\bm'\in \Ga}|V^{(\bk-\bm')}|^2(1+N_{\bm'}^2)^{-s}\right)\\
& \leq & ||U||_{s}^2 \left(\sum_{\bk\in \Ga}|V^{(\bk-\bm')}|^2\right) R_{s}
\end{eqnarray*}
where $R_s:=\sum_{\bm' \in \Ga}(1+N_{\bm'}^2)^{-s}$. This last sum converges if $s>q/2$. Indeed, according to \cite{iooss-rucklidge}[(24)], $\card\{\bk\in\Ga\;|\;N_{\bk}=N\}\leq c(q)N^{q-1}$ for some constant $c(q)$.  Hence
$\|UV\|_{0}^2 \leq \|U\|_{s}^2 \|V\|_{0}^2R_s.$
\end{proof}

\begin{lemma}(Moser-Nirenberg type inequality)\label{algebrab}
Let $s,s'>q/2$ and let $U,V\in \cH_{s}\cap\cH_{s'}$. Then,
\beq\label{algebra}
\|UV\|_{s} \leq C(s,s')(\|U\|_{s}\|V\|_{s'}+\|U\|_{s'}\|V\|_{s})
\eeq
for some positive constant $C(s,s')$ that depends only on $s$ and $s'$. For $\ell \geq 0$ and $s>\ell +q/2$, ${\mathcal{H}}_{s}$ is
continuously embedded into $\mathcal{C}^{\ell }$
\end{lemma}
We postpone the proof to the appendix.

\subsection{Formal computation}

\label{sec:formal}

Let us look for formal solutions of the steady Swift--Hohenberg equation 
\begin{equation}
\lambda U-(1+\Delta )^{2}U-U^{3}=0,  \label{eq:sh}
\end{equation}%
We characterise the functions of interest by their Fourier coefficients on
the \emph{quasilattice} $\Gamma $ generated by the $2q$ equally spaced unit
vectors $\mathbf{k}_{j}$ (see (\ref{defkm})):%
\begin{equation*}
U(\mathbf{x})=\sum_{\mathbf{k}\in \Gamma }u^{(\mathbf{k})}e^{i\mathbf{k}.%
\mathbf{x}},\text{ \ }\mathbf{x}=(x_{1},x_{2})\in 
\mathbb{R}
^{2}.
\end{equation*}

We seek a non trivial solution, bifurcating from $0$, parameterized by $%
\epsilon ,$ and which is invariant under rotations by $\pi /q$. As it is
shown for example in \cite{iooss-rucklidge}, a formal computation with
identification of orders in $\epsilon $ leads to 
\begin{equation}
U(x_{1},x_{2})=\epsilon u_{0}(x_{1},x_{2})+\epsilon
^{3}u_{1}(x_{1},x_{2})+\dots \qquad \lambda =\epsilon ^{2}\lambda
_{2}+\epsilon ^{4}\lambda _{4}+\dots  \label{eq:formal}
\end{equation}

and gives at order $\mathcal{O}(\epsilon )$ 
\begin{equation}
0=(1+\Delta )^{2}u_{0}.  \label{order1}
\end{equation}%
We take as our basic solution a quasipattern that is invariant under
rotations by $\pi /q$: 
\begin{equation}
u_{0}=\sum_{j=1}^{2q}e^{i\mathbf{k}_{j}\cdot \mathbf{x}}.  \label{U1}
\end{equation}

At order $\mathcal{O}(\epsilon ^{3})$ we have 
\begin{equation}
\lambda _{2}u_{0}-u_{0}^{3}=(1+\Delta )^{2}u_{1}.  \label{order3}
\end{equation}%
In order to solve this equation for $u_{1}$, we must impose a solvability
condition, namely that the coefficients of $e^{i\mathbf{k}_{j}\cdot \mathbf{x%
}}$, \ $j=1,...,2q$ on the left hand side of this equation must be zero.
Because of the invariance under rotations by $\pi /q,$ it is sufficient to
cancel the coefficient of $e^{i\mathbf{k}_{1}\cdot \mathbf{x}}.$ For the
computation of the coefficient, we need the following property

\textbf{Property: }\emph{If we have} 
\begin{equation*}
\mathbf{k}_{j}+\mathbf{k}_{l}+\mathbf{k}_{r}+\mathbf{k}_{s}=0\text{ for }%
j,l,r,s\in \{1,2q\}
\end{equation*}%
\emph{then either} $\mathbf{k}_{j}+\mathbf{k}_{l}=0,$ \emph{or} $\mathbf{k}%
_{j}+\mathbf{k}_{r},$ \emph{or} $\mathbf{k}_{j}+\mathbf{k}_{s}=0$ \emph{\
(there are two pairs of opposite unit vectors).}

\begin{proof} Since there are 4 unit vectors on the unit circle, we can assume without restriction, that $\mathbf{k}_{j}$ and 
$\mathbf{k}_{l}$ make an angle $2\theta \leq \pi /2.$ Then $|\mathbf{k}_{j}+\mathbf{k%
}_{l}|=2\cos \theta \geq \sqrt{2}.$ It results that $|\mathbf{k}_{r}+\mathbf{%
k}_{s}|=2\cos \theta $ with $\mathbf{k}_{r}$ and $\mathbf{k}_{s}$ symmetric
with respect to the direction of the bissectrix of $(\mathbf{k}_{j},\mathbf{k%
}_{l}),$ making the same angle as $\mathbf{k}_{j}$ and $\mathbf{k}_{l}$ with
the bissectrix.\ So $\{\mathbf{k}_{r}$, $\mathbf{k}_{s}\}$ is the symmetric
with respect to 0 of $\{\mathbf{k}_{j}$, $\mathbf{k}_{l}\}.$
\end{proof}

This yields 
\begin{equation}
\lambda _{2}=3(2q-1)  \label{mu2}
\end{equation}%
\emph{which is strictly positive}. Moreover 
\begin{eqnarray}
u_{1} &=&\sum_{\mathbf{k}\in \Gamma ,N_{\mathbf{k}}\neq 1,N_{\mathbf{k}}\leq 3}\alpha _{\mathbf{k}%
}e^{i\mathbf{k}\cdot \mathbf{x}},  \label{u1} \\
\alpha _{3\mathbf{k}_{j}} &=&-1/64,\text{ \ }\alpha _{2\mathbf{k}_{j}+%
\mathbf{k}_{l}}=-\frac{3}{(1-|2\mathbf{k}_{j}+\mathbf{k}_{l}|^{2})^{2}},%
\text{ \ }\mathbf{k}_{j}+\mathbf{k}_{l}\neq 0,  \notag \\
\alpha _{\mathbf{k}_{j}+\mathbf{k}_{l}+\mathbf{k}_{r}} &=&-\frac{6}{(1-|%
\mathbf{k}_{j}+\mathbf{k}_{l}+\mathbf{k}_{r}|^{2})^{2}},\text{ }j\neq l\neq
r\neq j,  \notag \\
\mathbf{k}_{j}+\mathbf{k}_{l} &\neq &0,\mathbf{k}_{j}+\mathbf{k}_{r}\neq 0,%
\mathbf{k}_{r}+\mathbf{k}_{l}\neq 0.  \notag
\end{eqnarray}%
We notice that for any $\mathbf{k}$, $\alpha _{\mathbf{k}}<0$ in $u_{1}$.
At order $\mathcal{O}(\epsilon ^{5})$ we have 
\begin{equation}
\lambda _{4}u_{0}+\lambda _{2}u_{1}-3u_{0}^{2}u_{1}=(1+\Delta )^{2}u_{2}.
\label{order5}
\end{equation}%
The solvability condition gives $\lambda _{4}$ equal to the coefficient of $%
e^{i\mathbf{k}_{1}\cdot \mathbf{x}}$ in $3u_{0}^{2}u_{1}$. From the
expression (\ref{u1}) of $u_{1}$ 
\begin{equation*}
\lambda _{4}=\sum_{\mathbf{k}_{j}+\mathbf{k}_{l}+\mathbf{k}=\mathbf{k}_{1},%
\text{ }N_{\mathbf{k}}=3}\alpha _{\mathbf{k}}
\end{equation*}%
where all coefficients $\alpha _{\mathbf{k}}$ are negative, it results that%
\begin{equation}
\lambda _{4}<0.  \label{mu4}
\end{equation}%

\subsection{Formulation of the problem}

Let us define the new unknown function $W$ in rewriting (\ref{eq:formal})
as: 
\begin{eqnarray}
U &=&U_{\epsilon }+\epsilon ^{4}W  \notag \\
U_{\epsilon } &=&\epsilon u_{0}+\epsilon ^{3}u_{1}+\epsilon
^{5}u_{2}  \label{eq:perturbed} \\
\lambda _{\epsilon } &=&\epsilon ^{2}\lambda _{2}+\epsilon ^{4}\lambda
_{4}  \notag
\end{eqnarray}%
where $u_{0}$, $u_{1}$, $u_{2},$ $\lambda _{2}$, $\lambda _{4},$ are as above. Given a particular (small) positive value of $%
\lambda $, we get $\epsilon ^{2}$ by the implicit function theorem, and
since $\lambda _{2}>0$, we obtain a unique positive $\epsilon $. All the
corrections are in $W$. The aim is to show that the quasi-periodic function $%
W$ exists and is small as $\epsilon $ tends towards $0$. By construction we
have%
\begin{equation*}
\lambda _{\epsilon }U_{\epsilon }-(1+\Delta )^{2}U_{\epsilon }-U_{\epsilon
}^{3}=:-\epsilon ^{7}f_{\epsilon }
\end{equation*}%
where $f_{\epsilon }$ is quasi-periodic, of order $\mathcal{O}(1)$ with a
finite expansion, and is function of $\epsilon ^{2}$. After substituting (%
\ref{eq:perturbed}) into the PDE (\ref{eq:sh}), we obtain an equation of the
form

\begin{equation*}
\mathcal{F}(\epsilon ,W)=0,
\end{equation*}%
with 
\begin{equation}
\mathcal{F}(\epsilon ,W)=:\mathcal{L}_{\epsilon }W+\epsilon ^{3}f_{\epsilon
}+3\epsilon ^{4}U_{\epsilon }W^{2}+\epsilon ^{8}W^{3},  \label{basic-equW}
\end{equation}%
where 
\begin{eqnarray}
\mathcal{L}_{\epsilon } &=&(1+\Delta )^{2}-\lambda _{\epsilon }+3U_{\epsilon
}{}^{2}=:{\mathfrak{L}}_{\epsilon }+\epsilon ^{6}\mathcal{P}_{\epsilon }, 
\label{eq:Lepsi} \\
{\mathfrak{L}}_{\epsilon } &=&(1+\Delta )^{2}+\epsilon ^{2}a+\epsilon ^{4}b,
\label{delL_ep} \\
a &=&3u_{0}^{2}-\lambda _{2},\text{ \ }b=6u_{0}u_{1}-\lambda _{4},  \notag \\
\mathcal{P}_{\epsilon } &=&6u_{0}u_{2}+3(u_{1}+\epsilon ^{2}u_{2})^{2}.  \notag
\end{eqnarray}%
\begin{remark}\label{degree-truncation}
The degree of truncation, that is the degree in $\epsilon$ in the expansion of $U_{\ep}$, is chosen so that the power of $\ep$ in front of both $W^2$ and $f_{\ep}$ are greater than $2$. This is crucial for the very last step of the proof.
\end{remark}

It is clear that the operator $\mathcal{P}_{\epsilon }$ is an operator
bounded in any $\mathcal{H}_{r}$, $r\geq 0$ uniformly bounded in $\epsilon $%
, for $\epsilon \leq \epsilon _{0}$.

A nice property of the operator ${\mathfrak{L}}_{\epsilon }$ is that the
averages $a_{0}$ and $b_{0}$ of $a$ and $b$ are strictly $>0$. Indeed, we
have for any $q$ 
\begin{equation*}
a_{0}=3,\text{ \ \ \ }b_{0}=-\lambda _{4}>0.
\end{equation*}%
For $a_{0}$ this results from (\ref{mu2}) and a simple examination of $%
u_{0}^{2}$ the average of which is $2q,$ and for $b_{0}$ we observe that the
average of $u_{0}u_{1}$ is 0, due to the form of $u_{1}.$

Assume that we could prove that ${\mathcal{L}}_{\epsilon }^{-1}$ is $%
\mathcal{O}(\epsilon ^{-2})$. Then, provided that we are in ${\mathcal{H}}%
_{s},s>q/2$ which is a Banach algebra, we should get from (\ref{basic-equW}) 
\begin{equation*}
W=\mathcal{O}(\epsilon)+\mathcal{O}(\epsilon ^{2}||W||^{2})
\end{equation*}%
The standard implicit function theorem then would allow to conclude and to
get $W=\mathcal{O}(\epsilon)$.

In fact, \emph{it is not expected that the operator $\mathcal{L}_{\epsilon }$ has
a bounded inverse, due to the small divisor problem mentioned in section
1.1}. Notice in particular that it is shown in \cite{iooss-rucklidge} that
there exists $c>0$ such that for any $\mathbf{k}\in \Gamma \backslash \{%
\mathbf{k}_{j};j=1,2,...,2q\}$%
\begin{equation}
\left\vert |\mathbf{k}|^{2}-1\right\vert \geq \frac{c}{N_{\mathbf{k}%
}^{2l_{0}}},  \label{dioph1}
\end{equation}%
where $l_{0}+1$ is the order of the algebraic number $\omega =2\cos \pi /q.$
This estimate is similar to the Siegel's diophantine condition for linearization of vector fields \cite{Arn2}.

This lower bound shows that the inverse of ${\mathfrak{L}}_{0}$ on the
orthogonal complement of its kernel is an unbounded operator in $\mathcal{H}%
_{s}$, only bounded from $\mathcal{H}_{s}$ to $\mathcal{H}_{s-4l_{0}}.$ In
other words, $0$ \emph{belongs to the continuous spectrum of }${\mathfrak{L}}_{0}$
and the main difficulty to be solved below is to find a bound for the
inverse ${\mathcal{L}}_{\epsilon }^{-1}$ for small values of $\epsilon .$
Notice that ${\mathcal{L}}_{\epsilon }$ is selfadjoint in $\mathcal{H}_{0}$
but not in $\mathcal{H}_{s}$ for $s>0.$ It is tempting to work on its small
(real) eigenvalues to obtain a bound of its inverse. However, we are in
infinite dimensions, so the spectrum does not contain only eigenvalues, and
an option would be to truncate the space to functions with finite Fourier
expansions (with $\mathbf{k}$ such that $N_{\mathbf{k}}\leq N).$ Since our
method consists in reducing the study to an operator in a smaller space, it
is preferable to use the eigenvalues later, on the reduced operator.

\subsection{Splitting of the space and first reduction of the problem\label%
{special splitting section}}

Let us split the space $\mathcal{H}_{s}$ into three mutually othogonal (in
any ${\mathcal{H}}_{s}$) subspaces. We define 
\begin{eqnarray*}
E_{0} &=&\left\{ W=\sum_{\mathbf{k}\in \Gamma }W^{(\mathbf{k})}e^{i\mathbf{k}%
\cdot \mathbf{x}}\in {\mathcal{H}}_{s};||\mathbf{k}|^{2}-1|^{2}\geq \delta
^{2},\text{ and }|\mathbf{k}-\mathbf{k}_{j}|>\delta _{1},,j\in
\{1,2q\}\right\} , \\
E_{1} &=&\left\{ W=\sum_{\mathbf{k}\in \Gamma }W^{(\mathbf{k})}e^{i\mathbf{k}%
\cdot \mathbf{x}}\in {\mathcal{H}}_{s};\mathbf{k}\in \sigma _{1}\right\} , \\
E_{2} &=&\left\{ W=\sum_{\mathbf{k}\in \Gamma }W^{(\mathbf{k})}e^{i\mathbf{k}%
\cdot \mathbf{x}}\in {\mathcal{H}}_{s};\exists j\in \{1,..2q\}\text{ such
that }\mathbf{k}\in \sigma _{2,j}\right\} ,
\end{eqnarray*}%
where (see figure \ref{fig:spectra})%
\begin{eqnarray*}
\sigma _{1} &=&\{\mathbf{k}\in \Gamma ;\text{ }||\mathbf{k}%
|^{2}-1|^{2}<\delta ^{2}\text{ and }|\mathbf{k}-\mathbf{k}_{j}|>\delta _{1}=%
\sqrt{3\delta },j\in \{1,2q\}\}, \\
\sigma _{2,j} &=&\{\mathbf{k}\in \Gamma ;\text{ }|\mathbf{k}-\mathbf{k}%
_{j}|\leq \delta _{1}\},\text{ }\sigma _{2}=\cup _{j=1}^{2q}\sigma _{2,j}.
\end{eqnarray*}%
\begin{figure}[tbp]
\begin{center}
\includegraphics[width=0.75\hsize]{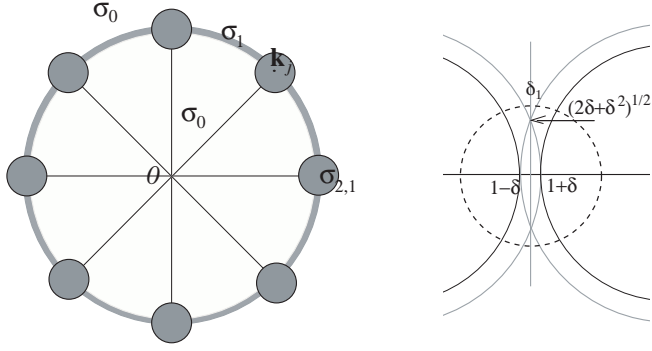}
\end{center}
\par
\caption{Division of the Fourier spectrum into $\protect\sigma _{0}$, $%
\protect\sigma _{1}$ and $\protect\sigma _{2}$.}
\label{fig:spectra}
\end{figure}
In figure \ref{fig:spectra} the annulus centered at the unit circle has a
thickness $2\delta $ and the little discs should have a radius $\delta _{1}=%
\sqrt{3\delta }$, such that the intersection of the shaded area with the
shifted one, centered at the point (2$\mathbf{k}_{1}$,0) is, for $\delta $
small enough, reduced to the disc centered at $\mathbf{k}_{1}.$ In the
sequel,\emph{\ we choose} $\delta =C\epsilon ^{1/2}$ \emph{with} $C$\emph{\
large enough} and 
\begin{equation*}
\delta _{1}=\epsilon ^{1/4}\sqrt{3C}
\end{equation*}%
hence $2\delta +\delta ^{2}<\delta _{1}^{2}$ is verified for $\epsilon
^{1/2}<1/C$ and the intersection $\sigma _{1}\cap \{\sigma _{1}+2\mathbf{k}%
_{1}\}$ is empty (\emph{hint:} solve $\delta _{1}^{2}=(1+\delta )^{2}-1$ for
intersecting the circle centered in $0,$ of radius $1+\delta $ with the line
of abscissa $1$ parallel to $y$ axis).

This leads to (see figure \ref{fig:spectra-intersec})%
\begin{equation*}
\sigma _{1}\cap \{\sigma _{1}+\mathbf{k}_{j}+\mathbf{k}_{l},\text{ }%
j,l=1,..2q,\text{ \ }\mathbf{k}_{j}+\mathbf{k}_{l}\neq 0\}=\varnothing .
\end{equation*}

\begin{figure}[tbp]
\includegraphics[width=1.0\hsize]{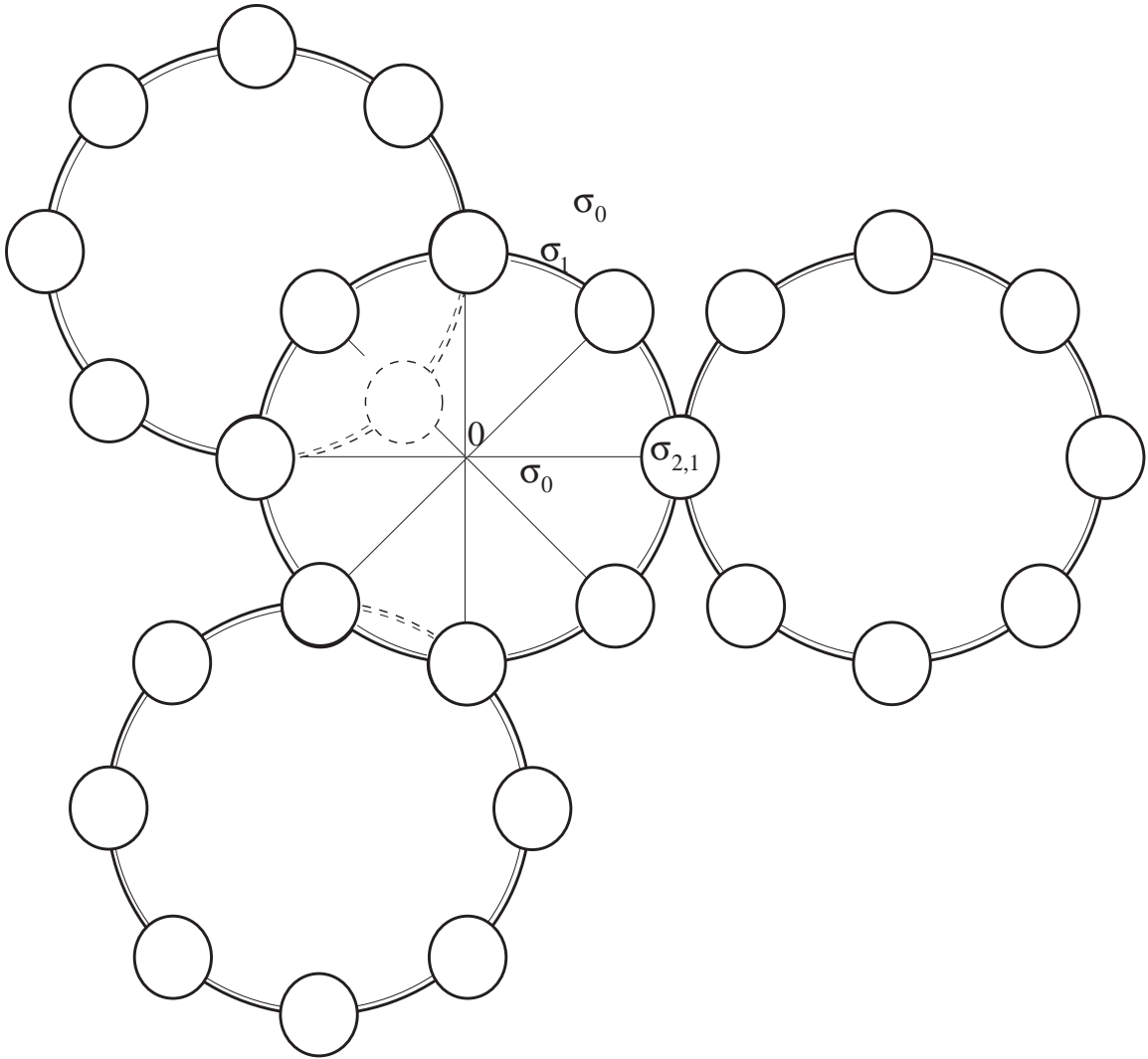}
\caption{Empty intersections of $\protect\sigma _{1}$ with \ $\protect\sigma %
_{1}$ shifted by $\mathbf{k}_{j}+\mathbf{k}_{l}$}
\label{fig:spectra-intersec}
\end{figure}
The subspaces $E_{l}$ are closed as intersections of closed subspaces
(kernel of certain coefficients, still continuous functionnals here) and we
have the orthogonal decomposition%
\begin{equation*}
{\mathcal{H}}_{s}=E_{0}\overset{\bot }{\oplus }E_{1}\overset{\bot }{\oplus }%
E_{2}.
\end{equation*}%
The orthogonal projections associated with this decomposition are denoted by 
$P_{0},P_{1},P_{2}.$ We also notice that the \emph{multiplication operator
by a function} having a \emph{finite Fourier expansion} with wave vectors in 
$\Gamma $ is a \emph{bounded} linear operator in ${\mathcal{H}}_{r}$ for any 
$r\geq 0$. Indeed a finite Fourier expansion belongs to ${\mathcal{H}}_{s}$
for any $s,$ and Lemmas \ref{algebraa} and \ref{algebrab} apply.

\subsection{Reduction to the subspace $E_{2}$}

\label{sec:bound}The aim here is to solve with respect to $U\in {\mathcal{H}}%
_{s}$ the equation%
\begin{equation}
\mathcal{L}_{\epsilon }U=f,  \label{equ for e.v of L_ep}
\end{equation}%
where $\mathcal{L}_{\epsilon }$ is defined in (\ref{eq:Lepsi}) and where $%
f\in {\mathcal{H}}_{s}$ is given.

\begin{remark}
We denote by $\mathcal{L}_{0}$ the operator $\mathcal{L}_{0,0}=\mathcal{L}%
_{0,W}=(1+\Delta )^{2}=\mathfrak{L}_0.$ For being more explicit in formulae,
we denote by $P_{j}\mathcal{L}_{0}P_{j}$ the restriction of $\mathcal{L}_{0}$
to the invariant subspace $E_{j},j=0,1,2.$
\end{remark}

We decompose this equation on the subspaces $E_{0},E_{1},E_{2},$ which
gives, after noticing that the subspaces $E_{0},$ $E_{1},$ $E_{2}$ are
invariant under $\mathcal{L}_{0},$ and that $\mathcal{L}_{0}|_{E_{1}}$, $%
\mathcal{L}_{0}|_{E_{2}}$ are bounded operators, and $P_{j}\mathcal{L}%
_{0}P_{0}=0$ for $j\in \{1,2\},$%
\begin{eqnarray}
P_{0}\mathcal{L}_{\epsilon }U_{0}+P_{0}\{(\epsilon ^{2}a+\epsilon
^{4}b+\epsilon ^{6}\mathcal{P}_{\epsilon })(U_{1}+U_{2})\} &=&f_{0}
\label{P0} \\
P_{1}\mathcal{L}_{\epsilon }U_{1}+P_{1}\{(\epsilon ^{2}a+\epsilon
^{4}b+\epsilon ^{6}\mathcal{P}_{\epsilon })(U_{0}+U_{2})\} &=&f_{1}
\label{P1} \\
P_{2}\mathcal{L}_{\epsilon }U_{2}+P_{2}\{(\epsilon ^{2}a+\epsilon
^{4}b+\epsilon ^{6}\mathcal{P}_{\epsilon })(U_{0}+U_{1})\} &=&f_{2}
\label{P2}
\end{eqnarray}%
where $f_{j}=P_{j}f,j=0,1,2.$

We have the following

\begin{lemma}
\label{LemFirstDecomp} Let fix $S\geq 0$ and choose $s$ such that $0\leq
s\leq S.$ Then, there exists $\epsilon _{0}>0,$ such that choosing $C$ large
enough in the definition of $\delta $, $\epsilon \leq \epsilon _{0}$, 
the equation $\mathcal{L}_{\epsilon }U=f$ in $\mathcal{H}_{s}$ reduces to 
\begin{equation*}
\mathcal{L}_{\epsilon }^{(2)}U_{2}=\mathcal{Q}(\epsilon )f
\end{equation*}%
where $U_{2}=P_{2}U$,
\begin{eqnarray}
\mathcal{L}_{\epsilon }^{(2)}&=&\Lambda _{\epsilon}+\epsilon ^{4}R_{\epsilon }\notag \\
\Lambda _{\epsilon }U_{2} &=&P_{2}(\mathcal{L}_{0}+\epsilon ^{2}a)U_{2} 
\notag \\
\mathcal{R}_{\epsilon }U_{2} &=&P_{2}\{bU_{2}-\widetilde{a}(P_{0}\mathcal{L}%
_{0}P_{0})^{-1}P_{0}(\widetilde{a}U_{2})\}+O(\epsilon ^{2})U_{2},
\label{def Lambda_ep}
\end{eqnarray}%
Here $O(\epsilon ^{2})$ denotes a bounded linear operator in ${\cal H}_s$ which norm is bounded by $c\ep^2$. 
The components $U_0:=P_{0}U$ and $U_1:=P_{1}U$ are functions of $U_2$ and $f$ and satisfy the following inequalities~:
\begin{eqnarray}
||U_0||_{s} &\leq &c\epsilon ^{2}||U_{2}||_{s}+\frac{c}{\epsilon }%
||(P_{0}+P_{1})f||_{s}  \notag \\
||U_1||_{s} &\leq &c\epsilon ^{4}||U_{2}||_{s}+\frac{c}{\epsilon ^{2}}%
||(\epsilon P_{0}+P_{1})f||_{s}  \notag
\end{eqnarray}%
where $\widetilde{a}=a-a_{0},$ for a certain $c>0,$ only depending on $s,$
and $R_{\epsilon }$ and $\mathcal{Q}(\epsilon )$ are bounded linear
operators in $\mathcal{H}_{s}$, depending smoothly on $\epsilon $.
\end{lemma}

We need the following Lemma

\begin{lemma}
\label{prop Proj}For $\epsilon $ small enough, we have%
\begin{equation}
\widetilde{a}U_{1}\in E_{0},\text{ \ }P_{1}(aU_{2})=0,\text{ }%
P_{1}(bU_{2})=0,  \label{propProj}
\end{equation}%
where $\widetilde{a}$ is the oscillating part of $a:$%
\begin{equation*}
\widetilde{a}=a-a_{0}.
\end{equation*}%
Moreover, the Fourier spectra of $P_{0}(aU_{2})$ and $P_{0}(bU_{2})$ are at
a distance of order 1 of the unit circle.
\end{lemma}

By distance of order 1 of the circle, we mean a strictly positive distance
as $\epsilon $ tends to 0, in the decomposition into subspaces $E_{j}$.

\begin{proof} We need to prove that i) for $%
\mathbf{k}\in \sigma _{1},$ then $\mathbf{k}+\mathbf{k}_{\mathbf{m}}\notin
\sigma _{1}\cup \sigma _{2}$ for $\mathbf{k}_{\mathbf{m}}\neq 0,|\mathbf{m}%
|=2,$ and ii) for $\mathbf{k}\in \sigma _{2},$ then $\mathbf{k}+\mathbf{k}_{%
\mathbf{m}}\notin \sigma _{1}$ with $|\mathbf{m}|=2$ or $4.$

For showing this, let us first observe that the intersections of the unit
circle with all circles of radius 1, centered at $\mathbf{k}_{\mathbf{m}},$
with $\mathbf{k}_{\mathbf{m}}\neq 0,$ $|\mathbf{m}|=2,$ are exactly the $2q$
points $\mathbf{k}_{r},$ \ $r=1,...,2q.$ Let us define the \emph{\ region }$%
\mathfrak{a}_{0}$ (shaded region in figure \ref{fig:spectra} (a)) defined by%
\begin{equation*}
\mathfrak{a}_{0}=\sigma _{1}\cup _{j\in \{1,..,2q\}}\sigma _{2,j}
\end{equation*}%
which is the union of the annulus $\sigma _{1}$ and the discs $\cup _{j\in
\{1,..,2q\}}\sigma _{2,j}.$ Now consider the intersection of $\mathfrak{a}%
_{0} $ with the union of shifted analogue annuli $\mathfrak{a}_{\mathbf{k}_{%
\mathbf{m}}}$ 
\begin{equation*}
\mathfrak{a}_{\mathbf{k}_{\mathbf{m}}}=\{\mathbf{k}+\mathbf{k}_{\mathbf{m}};%
\mathbf{k}\in \mathfrak{a}_{0},\text{ }\mathbf{k}_{\mathbf{m}}\neq 0,|%
\mathbf{m}|=2\}
\end{equation*}%
centered at all $\mathbf{k}_{\mathbf{m}},$such that $|\mathbf{m}|=2$ (see
figure \ref{fig:spectra-intersec}). It is clear that for any given $q,$ and $%
\epsilon $ small enough, the little discs of radius $\delta _{1}$ are such
that the intersection $\mathfrak{a}_{0}\cap \{\cup _{|\mathbf{m}|=2}%
\mathfrak{a}_{\mathbf{k}_{\mathbf{m}}}\}$ is exactly the union of the little
closed discs centered at each $\mathbf{k}_{j}.$ It results that for $U\in
E_{1}\oplus E_{2},$ the product $\widetilde{a}U$ for which the corresponding
wave vectors belong to some $\mathfrak{a}_{\mathbf{k}_{\mathbf{m}}},$ has a
zero projection on $E_{1}.$ This proves that%
\begin{equation*}
P_{1}(\widetilde{a}U_{1})=0,\text{ }P_{1}(\widetilde{a}U_{2})=0,
\end{equation*}%
which implies%
\begin{equation*}
P_{1}(aU_{2})=0.
\end{equation*}%
It is also clear that 
\begin{equation*}
\mathfrak{a}_{0}\cap \{\mathbf{k}+\mathbf{k}_{\mathbf{m}};\mathbf{k}\in
\sigma _{1},\text{ }\mathbf{k}_{\mathbf{m}}\neq 0,|\mathbf{m}%
|=2\}=\varnothing ,
\end{equation*}%
which moreover implies that%
\begin{equation*}
P_{2}(\widetilde{a}U_{1})=0,
\end{equation*}%
and (\ref{propProj}) is proved for the part concerning $a.$ Now observe that
we have $|\mathbf{k}_{j}+\mathbf{k}_{\mathbf{m}}|\neq 1$ except when $%
\mathbf{k}_{j}+\mathbf{k}_{\mathbf{m}}=\mathbf{k}_{r}$ for some $r\in
(1,2q). $ Since we only consider the finite number of cases $|\mathbf{m}|=2$
or $4,$ it is clear that in choosing $\epsilon $ small enough (i.e. $\delta
_{1}$ small enough), then for $\mathbf{k}\in \sigma _{2},$ $\mathbf{k}+%
\mathbf{k}_{\mathbf{m}}\notin \sigma _{1}.$ It results in particular that 
\begin{equation*}
P_{1}(bU_{2})=0.
\end{equation*}%
The last assertion of Lemma \ref{prop Proj} results from the fact that $|%
\mathbf{k}+\mathbf{k}_{\mathbf{m}}|\neq 1$ for the Fourier spectrum of terms 
$\in P_{0}(aU_{2})$ and $P_{0}(bU_{2}),$ with a distance to the unit circle
equivalent to $\left||\mathbf{k}_{j}+\mathbf{k}_{\mathbf{m}}|-1\right|$
when it is not zero. 
\end{proof}%
\begin{proof} \emph{of Lemma}
\ref{LemFirstDecomp} : We know by construction that 
\begin{equation*}
||(P_{0}\mathcal{L}_{0}P_{0})^{-1}||_{s}\leq \frac{1}{\epsilon C^{2}},
\end{equation*}%
Hence, we have%
\begin{equation}
(P_{0}\mathcal{L}_{\epsilon}P_{0})^{-1}=[1+(P_{0}\mathcal{L}%
_{0}P_{0})^{-1}(\epsilon ^{2}a+\epsilon ^{4}b+\epsilon ^{6}\mathcal{P}%
_{\epsilon})]^{-1}(P_{0}\mathcal{L}_{0}P_{0})^{-1},
\label{invResolventinE_0}
\end{equation}%
and the estimate%
\begin{equation*}
||\epsilon ^{2}a+\epsilon ^{4}b+\epsilon ^{6}\mathcal{P}_{\epsilon
}||_{s}\leq c(s)\epsilon ^{2}
\end{equation*}%
 leads for $%
\epsilon $ small enough ($s\leq S$) to 
\begin{eqnarray*}
(P_{0}\mathcal{L}_{\epsilon}P_{0})^{-1} &=&[\mathbb{I}-(P_{0}\mathcal{L}%
_{0}P_{0})^{-1}(\epsilon ^{2}a+\epsilon ^{4}b+\epsilon ^{6}\mathcal{P}%
_{\epsilon })+ \\
&&+\{(P_{0}\mathcal{L}_{0}P_{0})^{-1}(\epsilon ^{2}a+\epsilon ^{4}b+\epsilon
^{6}\mathcal{P}_{\epsilon })\}^{2}+O(\epsilon ^{3})](P_{0}\mathcal{L}%
_{0}P_{0})^{-1}
\end{eqnarray*}%
with a convergent Neumann power series in the bracket, as soon as $\frac{%
c(s)\epsilon }{C^{2}}<1,$ which holds for $\epsilon $ small enough ($s\leq S$%
). The first consequence is 
\begin{equation*}
||(P_{0}\mathcal{L}_{\epsilon }P_{0})^{-1}||_{s}\leq \frac{1}{\epsilon
(C^{2}-c(s)\epsilon )}.
\end{equation*}%
Notice that
\begin{equation}
U_0=-(P_{0}\mathcal{L}_{\epsilon }P_{0})^{-1}P_{0}\{(\epsilon ^{2}a+\epsilon
^{4}b+\epsilon ^{6}\mathcal{P}_{\epsilon })(U_1+U_{2})\}+(P_{0}\mathcal{L}_{\epsilon }P_{0})^{-1}f_0
\label{invU_0}
\end{equation}
The last property of Lemma \ref{prop Proj} and (\ref{invResolventinE_0}) imply that
in (\ref{invU_0}) we have, for $\epsilon \in (0,\epsilon _{0})$ and for any $s$%
\begin{eqnarray*}
||(P_{0}\mathcal{L}_{0}P_{0})^{-1}P_{0}(aU_{2})||_{s} &\leq
&d(s)||U_{2}||_{s}, \\
||(P_{0}\mathcal{L}_{0}P_{0})^{-1}P_{0}(bU_{2})||_{s} &\leq
&d(s)||U_{2}||_{s},
\end{eqnarray*}%
\begin{equation*}
||(P_{0}\mathcal{L}_{0}P_{0})^{-1}a(P_{0}\mathcal{L}%
_{0}P_{0})^{-1}P_{0}(aU_{2})||_{s}\leq d(s)||U_{2}||_{s}.
\end{equation*}%
It results that%
\begin{eqnarray*}
-(P_{0}\mathcal{L}_{\epsilon }P_{0})^{-1}P_{0}\{(\epsilon ^{2}a+\epsilon
^{4}b+\epsilon ^{6}\mathcal{P}_{\epsilon })U_{2}\} &=&-\epsilon ^{2}(P_{0}%
\mathcal{L}_{0}P_{0})^{-1}P_{0}(aU_{2})+ \\
&&+\mathcal{O}(\epsilon ^{4}||U_{2}||_{s}),
\end{eqnarray*}%
hence%
\begin{equation}
U_{0}=Q_{0,1}(\epsilon )U_{1}+Q_{0,2}(\epsilon )U_{2}+(P_{0}\mathcal{L}%
_{\epsilon }P_{0})^{-1}f_{0}  \label{U_0 fn of U_1 U-2}
\end{equation}%
with%
\begin{equation*}
Q_{0,j}(\epsilon )=:-(P_{0}\mathcal{L}_{\epsilon
}P_{0})^{-1}P_{0}(\epsilon ^{2}a+\epsilon ^{4}b+\epsilon ^{6}\mathcal{P}%
_{\epsilon })P_{j},\text{ \ }j=1,2,
\end{equation*}%
and for $s\leq S,$ $\epsilon \in (0,\epsilon _{0}(S))$ the following
estimates hold:%
\begin{equation*}
Q_{0,2}(\epsilon )U_{2}=-\epsilon ^{2}(P_{0}\mathcal{L}%
_{0}P_{0})^{-1}P_{0}(aU_{2})+\epsilon ^{4}Q_{0,2}^{\prime }(\epsilon
)U_{2},
\end{equation*}%
\begin{eqnarray}
||Q_{0,1}(\epsilon )U_{1}||_{s} &\leq &\frac{c(s)\epsilon }{%
C^{2}-c(s)\epsilon }||U_{1}||_{s}\leq c_{1}(s)\epsilon ||U_{1}||_{s},
\label{estimU_0} \\
||Q_{0,2}^{\prime }(\epsilon )U_{2}||_{s} &\leq &c_{1}(s)||U_{2}||_{s}, 
\notag \\
||(P_{0}\mathcal{L}_{\epsilon }P_{0})^{-1}f_{0}||_{s} &\leq &\frac{c_{1}(s)%
}{\epsilon }||f_{0}||_{s}.  \notag
\end{eqnarray}%
It results from Lemma \ref{prop Proj} that equation (\ref{P1}) leads to%
\begin{equation}
P_{1}(\mathcal{L}_{0}+\epsilon ^{2}a_{0}+\epsilon ^{4}b+\epsilon ^{6}%
\mathcal{P}_{\epsilon })U_{1}+P_{1}\{(\epsilon ^{2}a+\epsilon
^{4}b+\epsilon ^{6}\mathcal{P}_{\epsilon })U_{0}\}+P_{1}\epsilon ^{6}%
\mathcal{P}_{\epsilon }U_{2}=f_{1}  \label{equW_1}
\end{equation}%
and since $a_{0}=3>0$ and the operator $P_{1}\mathcal{L}_{0}P_{1}$ is
positive, we can invert the operator $P_{1}(\mathcal{L}_{0}+\epsilon
^{2}a_{0})P_{1}$ with the estimate%
\begin{equation*}
||\{P_{1}(\mathcal{L}_{0}+\epsilon ^{2}a_{0})P_{1}\}^{-1}||\leq \frac{1}{%
3\epsilon ^{2}}.
\end{equation*}%
Now replacing $U_{0}$ by its expression (\ref{U_0 fn of U_1 U-2}) into
equation (\ref{equW_1}), we introduce an operator acting on $U_{1}$ of the
form%
\begin{equation*}
P_{1}(\epsilon ^{4}b+\epsilon ^{6}\mathcal{P}_{\epsilon
})P_{1}+P_{1}(\epsilon ^{2}a+\epsilon ^{4}b+\epsilon ^{6}\mathcal{P}%
_{\epsilon })Q_{0,1}(\epsilon )
\end{equation*}%
which is bounded by $O(\epsilon ^{3}),$ perturbing $P_{1}(\mathcal{L}%
_{0}+\epsilon ^{2}a_{0})P_{1}$ the inverse of which is bounded by $%
1/3\epsilon ^{2}$. It results that, for $\epsilon $ small enough, the
operator acting on $U_{1}$ has a bounded inverse, with%
\begin{equation*}
||\{P_{1}[\mathcal{L}_{\epsilon }+(\epsilon ^{2}a+\epsilon ^{4}b+\epsilon
^{6}\mathcal{P}_{\epsilon })Q_{0,1}(\epsilon )]P_{1}\}^{-1}||_{s}\leq 
\frac{1}{\epsilon ^{2}}.
\end{equation*}%
Moreover by Lemma \ref{prop Proj} we have $P_{1}[a(P_{0}\mathcal{L}%
_{0}P_{0})^{-1}P_{0}(aU_{2})]=0,$ hence it results that there are bounded
linear operators
\begin{equation*}
\mathcal{B}_{\epsilon }^{(0)}:E_{2}\rightarrow E_{0},\text{ \ }\mathcal{B}%
_{\epsilon }^{(1)}:E_{2}\rightarrow E_{1},
\end{equation*}%
such that%
\begin{eqnarray}
U_{1} &=&\epsilon ^{4}\mathcal{B}^{(1)}(\epsilon )U_{2}+\frac{1}{\epsilon }%
\mathcal{Q}^{(1,0)}\mathcal{(}\epsilon )f_{0}+\frac{1}{\epsilon ^{2}}%
\mathcal{Q}^{(1,1)}\mathcal{(}\epsilon )f_{1},  \label{invW_1} \\
U_{0} &=&\epsilon ^{2}\mathcal{B}^{(0)}(\epsilon )U_{2}+\frac{1}{\epsilon }%
\mathcal{Q}^{(0,0)}\mathcal{(}\epsilon )f_{0}+\frac{1}{\epsilon }\mathcal{Q%
}^{(0,1)}\mathcal{(}\epsilon )f_{1},  \label{invW_0}
\end{eqnarray}%
with the estimates%
\begin{eqnarray*}
||\mathcal{B}^{(j)}(\epsilon )||_{s} &\leq &c_{2}(s),j=0,1, \\
||\mathcal{Q}^{(i,j)}\mathcal{(}\epsilon )||_{s} &\leq &c_{2}(s),\text{ \ }%
i,j=0,1,
\end{eqnarray*}%
uniform in $\epsilon \in (0,\epsilon _{0})$ and in $W$ bounded in $\mathcal{H%
}_{s}.$ \ Moreover, as $\epsilon \rightarrow 0$%
\begin{eqnarray*}
\mathcal{B}^{(0)}(\epsilon) &=&\mathcal{B}_{0}^{(0)}+\mathcal{O}(\epsilon
^{2}), \\
\mathcal{B}_{0}^{(0)}U_{2} &\sim &-(P_{0}\mathcal{L}%
_{0}P_{0})^{-1}P_{0}(\widetilde{a}U_{2}).
\end{eqnarray*}%
Equation (\ref{P2}) now reads%
\begin{equation*}
P_{2}\mathcal{L}_{\epsilon }U_{2}+P_{2}\{(\epsilon ^{2}a+\epsilon
^{4}b+\epsilon ^{6}\mathcal{P}_{\epsilon })U_{0}\}+P_{2}\{(\epsilon
^{4}b+\epsilon ^{6}\mathcal{P}_{\epsilon })U_{1}\}=f_{2}
\end{equation*}%
and replacing $U_{0}$ and $U_{1}$ by their expressions (\ref{invW_0}), (\ref%
{invW_1}) in function of $U_{2}$ leads to 
\begin{eqnarray*}
&&P_{2}\mathcal{L}_{\epsilon }U_{2}+P_{2}\{(\epsilon ^{2}a+\epsilon
^{4}b+\epsilon ^{6}\mathcal{P}_{\epsilon })\epsilon ^{2}\mathcal{B}%
^{(0)}(\epsilon )U_{2}+ \\
&&+P_{2}\{(\epsilon ^{4}b+\epsilon ^{6}\mathcal{P}_{\epsilon })\epsilon
^{4}\mathcal{B}^{(1)}(\epsilon )U_{2}\} \\
&=&f_{2}-P_{2}\left( (\epsilon ^{2}a+\epsilon ^{4}b+\epsilon ^{6}\mathcal{P}%
_{\epsilon })[\frac{1}{\epsilon }\mathcal{Q}^{(0,0)}\mathcal{(}\epsilon
)f_{0}+\frac{1}{\epsilon }\mathcal{Q}^{(0,1)}\mathcal{(}\epsilon
)f_{1}]\right) + \\
&&-P_{2}\left( (\epsilon ^{4}b+\epsilon ^{6}\mathcal{P}_{\epsilon })[\frac{%
1}{\epsilon }\mathcal{Q}^{(1,0)}\mathcal{(}\epsilon )f_{0}+\frac{1}{%
\epsilon ^{2}}\mathcal{Q}^{(1,1)}\mathcal{(}\epsilon )f_{1}]\right) ,
\end{eqnarray*}%
this gives (see the definition of $\Lambda _{\epsilon }$ in Lemma \ref%
{LemFirstDecomp}) 
\begin{equation*}
\Lambda _{\epsilon }U_{2}+\epsilon ^{4}R_{\epsilon} U_{2}=\mathcal{Q}%
(\epsilon )f
\end{equation*}%
with the announced properties for bounded operators $R_{\epsilon}$ and $%
\mathcal{Q}(\epsilon)$ in $\mathcal{H}_{s}.$ Lemma \ref{LemFirstDecomp}
is proved.
 \end{proof}

%

\subsection{Structure of the reduced operator $\Lambda _{\protect\epsilon }$}

\subsubsection{General structure - Invariant subspaces}

We study in this section the structure of the operator $\Lambda _{\epsilon }$
defined by (\ref{def Lambda_ep}). We first observe that we deal with
functions $U$ which are invariant under the rotation $R_{\frac{\pi }{q}}$ of
the plane. So, let us define a new subspace of $\mathcal{H}_{s}$ for such
functions:%
\begin{equation*}
E_{2}^{(S)}=\{U\in E_{2};U(R_{\frac{\pi }{q}}\mathbf{x})=U(\mathbf{x})\}.
\end{equation*}%
This implies immediately that%
\begin{equation}
U^{(\mathbf{k})}=U^{(R_{\frac{\pi }{q}}\mathbf{k})},  \label{symCoef}
\end{equation}%
and for any $U\in E_{2}^{(S)}$ we have the following decomposition%
\begin{equation*}
U=\sum_{j=1,...,2q}U_{2,j}
\end{equation*}%
where%
\begin{equation}
U_{2,j}(\mathbf{x})=\sum_{\mathbf{k}\in \sigma _{2,j}}U^{(\mathbf{k})}e^{i%
\mathbf{k}\cdot \mathbf{x}}=U_{2,1}(R_{\frac{(1-j)\pi }{q}}\mathbf{x}).
\label{W2,j}
\end{equation}%
It results that any $U\in E_{2}^{(S)}$ may be written as%
\begin{equation*}
U(\mathbf{x})=\sum_{j=1,...,2q}U_{2,1}(R_{\frac{(1-j)\pi }{q}}\mathbf{x}).
\end{equation*}%
Let us notice that in the little disc $\sigma _{2,1}$ we have%
\begin{equation*}
\sigma _{2,1}\ni \mathbf{k}=\mathbf{k}_{1}+\mathbf{k}^{\prime },|\mathbf{k}%
^{\prime }|\leq \delta _{1}=\epsilon ^{1/4}\sqrt{3C},
\end{equation*}%
and let decompose the discs $\sigma _{2,l}$ into $2q$ equal sectors $\mathbf{%
k}_{l}+\Sigma _{m},m=1,...,2q$ such that%
\begin{equation}
\Sigma _{m}=\left\{ \mathbf{k}^{\prime }\in \Gamma ;|\mathbf{k}^{\prime
}|\leq \delta _{1},\arg \mathbf{k}^{\prime }\in \lbrack \frac{(m-1)\pi }{q}-%
\frac{\pi }{2q},\frac{(m-1)\pi }{q}+\frac{\pi }{2q})\right\} .
\label{Sigma_m}
\end{equation}%
For any $\mathbf{k}^{\prime }\in \Sigma _{l}$, we define the set of $2q$
spectral points $\sigma _{\mathbf{k}^{\prime }}^{(l)}$ by%
\begin{equation*}
\sigma _{\mathbf{k}^{\prime }}^{(l)}=\left\{ \mathbf{k}=\mathbf{k}_{j}+%
\mathbf{k}^{\prime }\in \sigma _{2,j};\text{ }j=1,...,2q\right\} .
\end{equation*}%
The subspace of $\mathcal{H}_{s}$ associated with $\sigma _{\mathbf{k}%
^{\prime }}^{(l)}$ is denoted by $E_{2,\mathbf{k}^{\prime }}^{(l)},$ so that
any $U\in E_{2}^{(S)}$ may be written as%
\begin{eqnarray*}
U(\mathbf{x}) &=&\sum_{j=1,...,2q}\sum_{\mathbf{k}_{1}+\mathbf{k}^{\prime
}\in \sigma _{2,1}}U^{(\mathbf{k}_{1}+\mathbf{k}^{\prime })}e^{i(\mathbf{k}%
_{j}+R_{\frac{(j-1)\pi }{q}}\mathbf{k}^{\prime })\cdot \mathbf{x}} \\
&=&\sum_{l=1,...,2q}\sum_{\mathbf{k}^{\prime }\in \Sigma
_{l}}\sum_{j=1,...,2q}U^{(\mathbf{k}_{1}+R_{\frac{(1-j)\pi }{q}}\mathbf{k}%
^{\prime })}e^{i(\mathbf{k}_{j}+\mathbf{k}^{\prime })\cdot \mathbf{x}},
\end{eqnarray*}%
hence%
\begin{eqnarray*}
U &=&\sum_{l=1}^{2q}\sum_{\mathbf{k}^{\prime }\in \Sigma _{l}}U^{(l,\mathbf{k%
}^{\prime })},\text{ \ }U^{(l,\mathbf{k}^{\prime })}\in E_{2,\mathbf{k}%
^{\prime }}^{(l)}, \\
U^{(l,\mathbf{k}^{\prime })}(\mathbf{x}) &=&\sum_{j=1,...,2q}U^{(\mathbf{k}%
_{1}+R_{\frac{(1-j)\pi }{q}}\mathbf{k}^{\prime })}e^{i(\mathbf{k}_{j}+%
\mathbf{k}^{\prime })\cdot \mathbf{x}}=\sum_{j=1,...,2q}U^{(\mathbf{k}_{j}+%
\mathbf{k}^{\prime })}e^{i(\mathbf{k}_{j}+\mathbf{k}^{\prime })\cdot \mathbf{%
x}}
\end{eqnarray*}%
and any $U\in E_{2}^{(S)}$ is \emph{completely determined by the set of }$%
2q- $\emph{\ dimensional} $U^{(1,\mathbf{k}^{\prime })}\in E_{2,\mathbf{k}%
^{\prime }}^{(1)},$ $\mathbf{k}^{\prime }\in \Sigma _{1},$ identified with
the set of components%
\begin{equation*}
\{U^{(\mathbf{k}_{j}+\mathbf{k}^{\prime })}=U^{(\mathbf{k}_{1}+R_{\frac{%
(1-j)\pi }{q}}\mathbf{k}^{\prime })},\text{ \ }j=1,...,2q\}.
\end{equation*}%
Indeed we have%
\begin{equation*}
U^{(l,\mathbf{k}^{\prime })}(\mathbf{x})=U^{(l+1,R_{\frac{\pi }{q}}\mathbf{k}%
^{\prime })}(R_{\frac{\pi }{q}}\mathbf{x}),
\end{equation*}%
hence 
\begin{eqnarray*}
U^{(l,\mathbf{k}^{\prime })}(\mathbf{x}) &=&U^{(1,R_{(1-l)\frac{\pi }{q}}%
\mathbf{k}^{\prime })}(R_{(1-l)\frac{\pi }{q}}\mathbf{x}),\text{ \ }\mathbf{k%
}^{\prime }\in \Sigma _{l}, \\
U^{(l,R_{\frac{(l-1)\pi }{q}}\mathbf{k}^{\prime })}(\mathbf{x})
&=&\sum_{j=1,...,2q}U^{({\mathbf{k}_{1}+R_{\frac{(l-j)\pi }{q}}\mathbf{k}%
^{\prime }})}e^{i(\mathbf{k}_{j}+R_{\frac{(l-1)\pi }{q}}\mathbf{k}^{\prime
})\cdot \mathbf{x}},\text{ \ }\mathbf{k}^{\prime }\in \Sigma _{1},
\end{eqnarray*}%
where%
\begin{equation}
U^{(1,\mathbf{k}^{\prime })}(\mathbf{x})=\sum_{j=1,...,2q}U^{({\mathbf{k}%
_{1}+R_{\frac{(1-j)\pi }{q}}\mathbf{k}^{\prime }})}e^{i(\mathbf{k}_{j}+%
\mathbf{k}^{\prime })\cdot \mathbf{x}},\text{ \ }\mathbf{k}^{\prime }\in
\Sigma _{1},  \label{W(1,k')}
\end{equation}%
and we observe that the coordinates of $U^{(l,R_{\frac{(l-1)\pi }{q}}\mathbf{%
k}^{\prime })},\mathbf{k}^{\prime }\in \Sigma _{1},$ correspond to those 
\emph{shifted} of $U^{(1,\mathbf{k}^{\prime })}$. Moreover we have%
\begin{equation}
U(\mathbf{x})=\sum_{l=1}^{2q}\sum_{\mathbf{k}^{\prime }\in \Sigma _{1}}U^{(1,%
\mathbf{k}^{\prime })}(R_{(1-l)\frac{\pi }{q}}\mathbf{x}).  \label{W(x)}
\end{equation}%
From now on, we denote by $E_{2,\mathbf{k}^{\prime }}$ the previously
defined 2q-dimensional subspace $E_{2,\mathbf{k}^{\prime }}^{(1)}$.

Looking at the form of the operator $\Lambda _{\epsilon }$ we see that the
wave vector $\mathbf{k}$ of $U$ is shifted by $\mathbf{k}_{\mathbf{m}},|%
\mathbf{m}|=2$ at order $\epsilon ^{2},$ and $|\mathbf{m}|=4$ at order $%
\epsilon ^{4}.$ Now, we observe that for a fixed finite $|\mathbf{m}|$,\ if
the combination 
\begin{equation*}
\mathbf{k}_{\mathbf{m}}-(\mathbf{k}_{1}-\mathbf{k}_{j})
\end{equation*}%
is not 0, then it has a minimal length of order 1 as $\epsilon $ tends to $0$%
. It results that for $\mathbf{k}=\mathbf{k}_{1}+\mathbf{k}^{\prime }\in
\sigma _{2,1},$ and $\mathbf{l}=\mathbf{k}_{j}+\mathbf{l}^{\prime }\in
\sigma _{2,j}$ with 
\begin{equation*}
\mathbf{k}_{\mathbf{m}}-(\mathbf{k}-\mathbf{l})\neq 0
\end{equation*}%
then, for $\epsilon $ small enough%
\begin{equation*}
\mathbf{k}_{\mathbf{m}}-(\mathbf{k}-\mathbf{l})=\mathcal{O}(1).
\end{equation*}%
It results that the only possibility for going from $\mathbf{l}\in \sigma
_{2,j}$ to $\sigma _{2,1}$ is to add $\mathbf{k}_{\mathbf{m}}=\mathbf{k}_{1}+%
\mathbf{k}_{j+q}=\mathbf{k}_{1}-\mathbf{k}_{j}$. It results that the system%
\begin{equation*}
\mathbf{k}_{\mathbf{m}}+\mathbf{l=k},\text{ }\mathbf{l}=\mathbf{k}_{j}+%
\mathbf{l}^{\prime },\text{ }\mathbf{k}=\mathbf{k}_{j}+\mathbf{k}^{\prime }
\end{equation*}%
has the only solution%
\begin{equation}
\mathbf{l}^{\prime }=\mathbf{k}^{\prime }.  \label{latticeAlgebra}
\end{equation}%
It should be clear that $\mathbf{k}_{\mathbf{m}}$ comes from terms with many
possible combinations, not only trivial ones as for $q=4.$ For example for $%
q=6,$ the terms occuring in the coefficients giving $\mathbf{k}_{\mathbf{m}}$
are even more frequent at order $\epsilon ^{4}$ because of the existing
special combinations $\mathbf{k}_{j}+R_{\pi /3}\mathbf{k}_{j}+R_{2\pi /3}%
\mathbf{k}_{j}=0.$ However, \emph{in all cases} we can write, for $\mathbf{k}%
=\mathbf{k}_{1}+\mathbf{k}^{\prime }\in \sigma _{2,1}$%
\begin{equation}
(\Lambda _{\epsilon }U)^{(\mathbf{k}_{1}+\mathbf{k}^{\prime
})}=\sum_{j=1}^{Q}\gamma _{j}(\mathbf{k},\epsilon )U^{(\mathbf{k}_{j}+%
\mathbf{k}^{\prime })}.  \label{LambdaW_k}
\end{equation}
\begin{remark}
The argument $\epsilon$ in $\gamma_{j}(\mathbf{k},\epsilon)$ only refers to the perturbation of $P_{2}\mathcal{L}_{0}P_{2}$ in $\Lambda_{\epsilon}$ (see (\ref{def Lambda_ep})), 
and not on the fact that $P_{2}$ also depends on 
$\epsilon$ via the radii of the little discs composing the set $\sigma_{2}$ 
which are $O(\epsilon^{1/4})$.
\end{remark}
\begin{remark}
Due to the form of orders $\epsilon ^{2}$ and $\epsilon ^{4}$ in $\Lambda
_{\epsilon }$ and $\mathcal{L}_{\epsilon }^{(2)}$, and because of (\ref%
{latticeAlgebra}), we notice that the dependency in $\mathbf{k}$ of the
coefficients $\gamma _{j}(\mathbf{k},\epsilon )$ only occurs at orders $%
\epsilon ^{0}$ and $\epsilon ^{4}$. Indeed, the dependency in $\mathbf{k}$
comes from operators $\mathcal{L}_{0}$ at order 0 and $(P_{0}\mathcal{L}%
_{0}P_{0})^{-1}$ in the term $P_{2}\{\widetilde{a}(P_{0}\mathcal{L}%
_{0}P_{0})^{-1}P_{0}(\widetilde{a}U)\}$ at order $\epsilon ^{4}.$
\end{remark}

The property that $\Lambda _{\epsilon }U$ is invariant under the rotation $%
R_{\frac{\pi }{q}}$ and the identity 
\begin{equation*}
R_{\frac{-\pi }{q}}\widetilde{\mathbf{k}}+\mathbf{k}_{j}-\mathbf{k}_{1}=R_{%
\frac{-\pi }{q}}(\widetilde{\mathbf{k}}+\mathbf{k}_{j+1}-\mathbf{k}_{2})
\end{equation*}%
lead to%
\begin{equation*}
(\Lambda _{\epsilon }U)^{(\widetilde{\mathbf{k}})}=\sum_{j=1}^{2q}\gamma
_{j}(R_{\frac{-\pi }{q}}\widetilde{\mathbf{k}},\epsilon )U^{(\widetilde{%
\mathbf{k}}+\mathbf{k}_{j+1}-\mathbf{k}_{2})},\text{ }\widetilde{\mathbf{k}}%
=R_{\frac{\pi }{q}}\mathbf{k}\in \sigma _{2,2}.
\end{equation*}%
Choosing $\sigma _{2,2}\ni \widetilde{\mathbf{k}}=\mathbf{k}^{\prime }+%
\mathbf{k}_{2},$ $\mathbf{|k}^{\prime }|\leq \delta _{1},$ we then have%
\begin{equation}
(\Lambda _{\epsilon }U)^{(\mathbf{k}^{\prime }+\mathbf{k}_{2})}=%
\sum_{j=1}^{2q}\gamma _{j}(R_{\frac{-\pi }{q}}\mathbf{k}^{\prime }+\mathbf{k}%
_{1},\epsilon )U^{(\mathbf{k}^{\prime }+\mathbf{k}_{j+1})}.
\label{LambdaW_k+k_2-K_1}
\end{equation}%
In the same way, after identifying $j+2q$ with $j,$ we obtain for $%
r=1,...,2q $ 
\begin{eqnarray*}
(\Lambda _{\epsilon }U)^{(\mathbf{k}^{\prime }+\mathbf{k}_{r})}
&=&\sum_{j=1}^{2q}\gamma _{j}\left( R_{\frac{\pi (1-r)}{q}}\mathbf{k}%
^{\prime }+\mathbf{k}_{1},\epsilon \right) U^{(\mathbf{k}^{\prime }+\mathbf{k%
}_{j+r-1})}, \\
&=&\sum_{j=1}^{2q}\gamma _{j+1-r}\left( R_{\frac{\pi (1-r)}{q}}\mathbf{k}%
^{\prime }+\mathbf{k}_{1},\epsilon \right) U^{(\mathbf{k}^{\prime }+\mathbf{k%
}_{j})}.
\end{eqnarray*}%
The important result is that, for a fixed $\mathbf{k}=\mathbf{k}_{1}+\mathbf{%
k}^{\prime },\mathbf{k}^{\prime }\in \Sigma _{1},$ the subspace $E_{2,%
\mathbf{k}^{\prime }}$ is invariant under the operator $\Lambda _{\epsilon }$
then denoted $\Lambda _{\epsilon }^{(\mathbf{k}^{\prime })}.$ Hence the $%
2q\times 2q$ \emph{matrices of }$\Lambda _{\epsilon }^{(\mathbf{k}^{\prime
})}$ \emph{are uncoupled for different} $\mathbf{k}^{\prime }\in \Sigma
_{1}. $ We notice that if $\gamma _{j}$ were independent of $\mathbf{k},$
the lines of the matrix of $\Lambda _{\epsilon }^{(\mathbf{k}^{\prime })}$
would be deduced each from the previous one by a simple right shift.

The next useful property of $\Lambda _{\epsilon }$ is its \emph{%
self-adjointness} in $E_{2}$ with the Hilbert structure of $\mathcal{H}_{0}$%
. This property is immediate from the definition (\ref{def Lambda_ep}) with
the scalar product of the space $\mathcal{H}_{s}$ for $s=0.$ It should be
noticed that the full linear operator $\mathcal{L}_{\epsilon ,W}^{(2)}$
acting in $E_{2}$ \emph{is not selfadjoint in general}. Now, isolating the
coordinates $U^{(\mathbf{k}^{\prime }+\mathbf{k}_{j})},$ $j=1,...,2q$ for $%
\mathbf{k}^{\prime }\in \Sigma _{1},$ we still have, for any fixed $\mathbf{k%
}^{\prime }$, a $2q\times 2q$ self-adjoint matrix $\Lambda _{\epsilon }^{(%
\mathbf{k}^{\prime })}$ due to the previous self-adjointness of the operator 
$\Lambda _{\epsilon }$ in $\mathcal{H}_{0}.$ It results that we have%
\begin{equation}
\gamma _{j+1-r}\left( R_{\frac{\pi (1-r)}{q}}\mathbf{k}^{\prime }+\mathbf{k}%
_{1},\epsilon \right) =\gamma _{r+1-j}\left( R_{\frac{\pi (1-j)}{q}}\mathbf{k%
}^{\prime }+\mathbf{k}_{1},\epsilon \right) .  \label{symGamma_j}
\end{equation}%
We sum up these results in the following

\begin{lemma}
\label{structure}The subspace $E_{2}^{(S)}$ of $\mathcal{H}_{s}$ consisting
of functions invariant under rotations by $\pi /q$ may be decomposed into
the following Hilbert sum%
\begin{equation*}
E_{2}^{(S)}=\underset{\mathbf{k}^{\prime }\in \Sigma _{1}}{\overset{\bot }{%
\oplus }}E_{2,\mathbf{k}^{\prime }}
\end{equation*}%
where we identify the wave vector $\mathbf{k}$ with $R_{\frac{\pi }{q}}%
\mathbf{k}$ i.e., $\mathbf{k}_{1}+\mathbf{k}^{\prime }\in \sigma _{2,1}$
with $\mathbf{k}_{j}+R_{\frac{\pi (j-1)}{q}}\mathbf{k}^{\prime }\in \sigma
_{2,j}$ (see (\ref{W(x)})).

The $2q-$ dimensional subspace $E_{2,\mathbf{k}^{\prime }}$ is invariant
under the operator $\Lambda _{\epsilon }.$ Defining the coefficients $\gamma
_{j}(\mathbf{k},\epsilon )$ by%
\begin{equation*}
(\Lambda _{\epsilon }U)^{(\mathbf{k}^{\prime }+\mathbf{k}_{1})}=%
\sum_{j=1}^{2q}\gamma _{j}(\mathbf{k},\epsilon )U^{(\mathbf{k}^{\prime }+%
\mathbf{k}_{j})},\text{ }\mathbf{k}\in \sigma _{2,1}
\end{equation*}%
the $2q\times 2q$ matrix of the restriction $\Lambda _{\epsilon }^{(\mathbf{k%
}^{\prime })}$ of $\Lambda _{\epsilon }$ to $E_{2,\mathbf{k}^{\prime }}$ is
symmetric and satisfies%
\begin{equation*}
\gamma _{j+1-r}\left( R_{\frac{\pi (1-r)}{q}}\mathbf{k}^{\prime }+\mathbf{k}%
_{1},\epsilon \right) =\gamma _{r+1-j}\left( R_{\frac{\pi (1-j)}{q}}\mathbf{k%
}^{\prime }+\mathbf{k}_{1},\epsilon \right)
\end{equation*}%
for any $\mathbf{k}^{\prime }\in \Sigma _{1}.$
\end{lemma}

Let us define $\La_{0}^{(\bk')}$ which is a \emph{%
diagonal matrix} with%
\begin{equation*}
\gamma _{1}(\mathbf{k},0)=(|\mathbf{k}|^{2}-1)^{2},\gamma _{2}(\mathbf{k}%
,0)=0,....\gamma _{2q}(\mathbf{k},0)=0.
\end{equation*}%
For $\mathbf{k}=\mathbf{k}^{\prime }+\mathbf{k}_{j}\in \sigma _{2,j},$ we
define 
\begin{equation}\label{beta}
\beta _{j}(\mathbf{k}^{\prime })=(|\mathbf{k}^{\prime }+\mathbf{k}%
_{j}|^{2}-1)^{2}, \quad j=1,...,2q.
\end{equation}
 Hence for $j=1,...,2q$%
\begin{eqnarray*}
\gamma _{1}\left( R_{\frac{\pi (1-j)}{q}}\mathbf{k}^{\prime }+\mathbf{k}%
_{1},0\right) &=&\left( \left\vert R_{\frac{\pi (1-j)}{q}}\mathbf{k}^{\prime
}+\mathbf{k}_{1}\right\vert ^{2}-1\right) ^{2}=\beta _{j}(\mathbf{k}^{\prime
}) \\
&=&(2\mathbf{k}_{j}\cdot \mathbf{k}^{\prime }+|\mathbf{k}^{\prime
}|^{2})^{2},
\end{eqnarray*}%
and $\Lambda _{0}^{(\mathbf{k}^{\prime })}$ reads%
\begin{equation*}
\Lambda _{0}^{(\mathbf{k}^{\prime })}=\left( 
\begin{array}{cccccccc}
\beta _{1}(\mathbf{k}^{\prime }) & 0 & . & . & 0 & . & 0 & 0 \\ 
0 & \beta _{2}(\mathbf{k}^{\prime }) & . & . & 0 & . & 0 & 0 \\ 
. & . & . & . & . & . & . & . \\ 
. & . & . & . & . & . & . & . \\ 
0 & 0 & . & . & \beta _{j}(\mathbf{k}^{\prime }) & . & . & 0 \\ 
. & . & . & . & . & . & . & . \\ 
0 & 0 & . & . & 0 & . & \beta _{2q-1}(\mathbf{k}^{\prime }) & 0 \\ 
0 & 0 & . & . & 0 & . & 0 & \beta _{2q}(\mathbf{k}^{\prime })%
\end{array}%
\right) .
\end{equation*}%
Then, according to the definition of $\Lambda_{\epsilon}^{(\mathbf{k}')}$, we
can write
\begin{equation*}
\Lambda_{\epsilon}^{(\mathbf{k}')}=\Lambda_{0}^{(\mathbf{k}')}+\epsilon^2 \Lambda_1.
\end{equation*}
According to (\ref{symGamma_j}), in the case when the coefficients $\gamma
_{j}(\mathbf{k},\epsilon )$ are independent of $\mathbf{k}$, (which
corresponds here to the order $\epsilon ^{2}$), this leads to a first line for
the $2q\times 2q$ matrix, of the form%
\begin{equation}
\gamma _{1},\gamma _{2},...\gamma _{q},\gamma _{q+1},\gamma _{q},\gamma
_{q-1},..\gamma _{3},\gamma _{2}  \label{simplform-gamma_j}
\end{equation}%
and next lines are deduced by a right shift, making a symmetric matrix. For
example in the case $q=4,$ we obtain for
 $\Lambda _{1 }$ a matrix of the form
(easily generalizable for any $q$) 
\begin{equation}
\left( 
\begin{array}{cccccccc}
\gamma _{1} & \gamma _{2} & \gamma _{3} & \gamma _{4} & \gamma _{5} & \gamma
_{4} & \gamma _{3} & \gamma _{2} \\ 
\gamma _{2} & \gamma _{1} & \gamma _{2} & \gamma _{3} & \gamma _{4} & \gamma
_{5} & \gamma _{4} & \gamma _{3} \\ 
\gamma _{3} & \gamma _{2} & \gamma _{1} & \gamma _{2} & \gamma _{3} & \gamma
_{4} & \gamma _{5} & \gamma _{4} \\ 
\gamma _{4} & \gamma _{3} & \gamma _{2} & \gamma _{1} & \gamma _{2} & \gamma
_{3} & \gamma _{4} & \gamma _{5} \\ 
\gamma _{5} & \gamma _{4} & \gamma _{3} & \gamma _{2} & \gamma _{1} & \gamma
_{2} & \gamma _{3} & \gamma _{4} \\ 
\gamma _{4} & \gamma _{5} & \gamma _{4} & \gamma _{3} & \gamma _{2} & \gamma
_{1} & \gamma _{2} & \gamma _{3} \\ 
\gamma _{3} & \gamma _{4} & \gamma _{5} & \gamma _{4} & \gamma _{3} & \gamma
_{2} & \gamma _{1} & \gamma _{2} \\ 
\gamma _{2} & \gamma _{3} & \gamma _{4} & \gamma _{5} & \gamma _{4} & \gamma
_{3} & \gamma _{2} & \gamma _{1}%
\end{array}%
\right) .  \label{structureLambda}
\end{equation}%
where $\gamma _{1},...\gamma _{q+1}$ are independent of $\mathbf{k}^{\prime
} $.

\subsubsection{Computation of coefficients $\protect\gamma _{j}$ in $\Lambda
_{1}$}

Let us compute the operator
\begin{eqnarray*}
E_{2}^{(S)} &\ni &U\mapsto \epsilon ^{2}P_{2}(aU), \\
a &=&3u_{0}^{2}-3(2q-1).
\end{eqnarray*}%
where $u_{0}$ is given by (\ref{U1}). We have for $U\in E_{2}^{(S)}$ (see(%
\ref{W2,j})) 
\begin{equation*}
P_{2,1}(u_{0}^{2}U)=2qU_{2,1}(\mathbf{x})+U_{2,q+1}(\mathbf{x})e^{2i\mathbf{k%
}_{1}\cdot \mathbf{x}}+\sum_{j=2,..,q,q+2,...,2q}2U_{2,j}(\mathbf{x})e^{i(%
\mathbf{k}_{1}-\mathbf{k}_{j})\cdot \mathbf{x}},
\end{equation*}%
where we denote by $P_{2,1}$ the orthogonal projection corresponding to the
part $\sigma _{2,1}$ of the spectrum. Since we have%
\begin{equation*}
U_{2,j}(\mathbf{x})e^{i(\mathbf{k}_{1}-\mathbf{k}_{j})\cdot \mathbf{x}%
}=\sum_{\mathbf{k}\in \sigma _{2,j}}U^{(\mathbf{k})}e^{i(\mathbf{k}+\mathbf{k%
}_{1}-\mathbf{k}_{j})\cdot \mathbf{x}}=\sum_{\mathbf{k}\in \sigma _{2,1}}U^{(%
\mathbf{k}+\mathbf{k}_{j}-\mathbf{k}_{1})}e^{i\mathbf{k}\cdot \mathbf{x}}
\end{equation*}%
we obtain%
\begin{equation}
P_{2,1}(aU)=\sum_{\mathbf{k}\in \sigma _{2,1}}3e^{i\mathbf{k}\cdot \mathbf{x}%
}\left\{ U^{(\mathbf{k})}+U^{(\mathbf{k}-2\mathbf{k}_{1})}+%
\sum_{j=2,..,q,q+2,...,2q}2U^{(\mathbf{k}+\mathbf{k}_{j}-\mathbf{k}%
_{1})}\right\} .  \label{P2(aW)}
\end{equation}%
As expected, it appears that the linear operator 
\begin{equation}
U\mapsto P_{2}(\epsilon ^{2}a)U,\text{ }U\in E_{2}^{(S)}  \label{reduced Op}
\end{equation}%
leaves invariant the subspaces $E_{2,\mathbf{k}^{\prime }}$ and in this
subspace it takes the form of a matrix with 4 identical blocks for the set
of $2q$ coordinates $U^{(1,\mathbf{k}^{\prime })}$, and coefficients $\gamma
_{j}$ are independent of $\mathbf{k}^{\prime }$ and have the form (\ref%
{simplform-gamma_j}) with:%
\begin{equation*}
\gamma _{1}=\gamma _{1+q}=3,\text{ \ }\gamma _{2}=\gamma _{3}=...\gamma
_{q}=6.
\end{equation*}%
Hence, {\bf $\Lambda_{1}^{(\bk')}$ is independent of $\bk'$ and we write $\Lambda_{1}^{(\bk')}=\Lambda
_{1}$}.

For example, in the case $q=4,$ we have the following corresponding matrix
for%
\begin{equation*}
U\mapsto (P_{2}(a\cdot )^{(\mathbf{k}^{\prime })})U,\text{ }U\in E_{2,%
\mathbf{k}^{\prime }}
\end{equation*}%
\begin{equation*}
\Lambda _{1}=3\left( 
\begin{array}{cccccccc}
1 & 2 & 2 & 2 & 1 & 2 & 2 & 2 \\ 
2 & 1 & 2 & 2 & 2 & 1 & 2 & 2 \\ 
2 & 2 & 1 & 2 & 2 & 2 & 1 & 2 \\ 
2 & 2 & 2 & 1 & 2 & 2 & 2 & 1 \\ 
1 & 2 & 2 & 2 & 1 & 2 & 2 & 2 \\ 
2 & 1 & 2 & 2 & 2 & 1 & 2 & 2 \\ 
2 & 2 & 1 & 2 & 2 & 2 & 1 & 2 \\ 
2 & 2 & 2 & 1 & 2 & 2 & 2 & 1%
\end{array}%
\right) .
\end{equation*}%
for each fixed $\mathbf{k}^{\prime }\in \Sigma _{1}$.

\subsection{Eigenvalues of $\Lambda _{\protect\epsilon }$}

From (\ref{def Lambda_ep}) we have%
\begin{equation*}
\Lambda _{\epsilon }^{(\mathbf{k}^{\prime })}=\Lambda _{0}^{(\mathbf{k}%
^{\prime })}+\epsilon ^{2}\Lambda _{1}.
\end{equation*}%
We show below the following

\begin{lemma}
\label{eigenvLambda}For any given $q\geq 4,$ and $\mathbf{k}\in \sigma _{2}$
the eigenvalues $\mu_j$of $\Lambda _{\epsilon }^{(\mathbf{k}^{\prime })}$ take the
form%
\begin{eqnarray*}
\mu _{j} &=&(2\mathbf{k}_{j}\cdot \mathbf{k}^{\prime }+|\mathbf{k%
}^{\prime }|^{2})^{2}+3\epsilon ^{2}+O(\epsilon ^{4}), \\
&=&[|\mathbf{k}^{\prime }+\mathbf{k}_{j}|^{2}-1]^{2}+3\epsilon
^{2}+O(\epsilon ^{4}),\text{ \ }j=1,..,2q.
\end{eqnarray*}%
\begin{proof} 
The eigenvalues $\mu \in 
\mathbb{R}
$ of $\Lambda _{\epsilon }^{(\mathbf{k}^{\prime })}$ satisfy for a certain $%
\zeta \in 
\mathbb{R}
^{2q}$%
\begin{equation}
\{\Lambda _{0}^{(\mathbf{k}^{\prime })}+\epsilon ^{2}\Lambda _{1}\}\zeta =\mu \zeta .
\label{EquEigenvalues}
\end{equation}%
Since we deal with selfadjoint operators, any eigenvalue takes the form (see 
\cite{kato-book} in the $2q$-dimensional subspace $E_{2,\mathbf{k}^{\prime }}
$.)%
\begin{equation*}
\mu _{j}=\mu _{j,0}(\mathbf{k}^{\prime })+\epsilon ^{2}\mu _{j,1}(%
\mathbf{k}^{\prime })+O(\epsilon ^{4}),\text{ \ }j=1,...,2q
\end{equation*}%
with%
\begin{equation*}
\mu _{j,0}(\mathbf{k}^{\prime })=(2\mathbf{k}_{j}\cdot \mathbf{k}^{\prime
}+|\mathbf{k}^{\prime }|^{2})^{2}=\beta_{j}(\bk')
\end{equation*}%
by definition $(\ref{beta})$.
Eigenvectors take the form%
\begin{equation*}
\zeta _{j}=\zeta _{j,0}+\epsilon ^{2}\zeta _{j,1}+O(\epsilon ^{4}),j=1,...,2q,
\end{equation*}%
with%
\begin{equation*}
\zeta _{j,0}=(0,..,0,1,0..,0)^{t},\text{ }1\text{ taking the }j\text{th
place.}
\end{equation*}%
A simple identification at order $\epsilon ^{2}$ leads to%
\begin{equation}
(\Lambda _{0}^{(\mathbf{k}^{\prime })}-\mu _{j,0})\zeta _{j,1}+(\Lambda
_{1}-\mu _{j,1})\zeta _{j,0}=0,  \label{ident eps3}
\end{equation}%
Taking the scalar product of (\ref{ident eps3}) with $\zeta _{j,0}$ gives,
taking into account the form of $\Lambda _{1},$%
\begin{equation*}
\mu _{j,1}=\frac{\langle \Lambda _{1}\zeta _{j,0},\zeta _{j,0}\rangle }{%
\langle \zeta _{j,0},\zeta _{j,0}\rangle }=3,\text{ \ }j=1,...,2q,
\end{equation*}%
which is independent of $\mathbf{k}^{\prime },$ and which gives the result
of Lemma \ref{eigenvLambda}. 
\end{proof}
\end{lemma}

\subsection{Inverse of $\mathcal{L}_{\protect\epsilon }$ in $\mathcal{H}_{s}$%
}

We already have the following estimate in $\mathcal{H}_{0}:$

\begin{lemma}
\label{estimInvLambda} For any given $q\geq 4,$ and for $\epsilon $ small
enough, the linear operator $\Lambda _{\epsilon }$ is invertible in $%
\mathcal{H}_{0}$ with 
\begin{equation*}
||\Lambda _{\epsilon }^{-1}||_{0}\leq \frac{1}{2\epsilon ^{2}}.
\end{equation*}
\end{lemma}

The proof of Lemma \ref{estimInvLambda} follows directly from Lemma \ref%
{eigenvLambda} for $\epsilon \leq \epsilon _{0}$, since $\mathbf{k}^{\prime
} $ is bounded, and all eigenvalues for $\mathbf{k}^{\prime }\in \Sigma _{1}$
are positive and larger than $2\epsilon ^{2}$.

For extending the estimate to $\mathcal{H}_{s},$ we need next property

\begin{lemma}\label{little lemma}
For any $K>0$, $|x-y|\leq K,$ $x$ and $y>0,$ and any $p\geq 0$ there exists $d(p,K)>0$
such that%
\begin{equation*}
|(1+x)^{p}-(1+y)^{p}|\leq d(p,K)(1+x)^{p-1}.
\end{equation*}
\end{lemma}
The proof of this Lemma is in Appendix A.

Then we prove the following
\begin{lemma}
\label{estimInvH_s}For any given $q\geq 4,$ and for $\epsilon $ small
enough, the linear operator $\Lambda _{\epsilon }$ is invertible in $%
\mathcal{H}_{s}$ for $s \geq 0$, with $c_{s}>0$ such that 
\begin{equation*}
||\Lambda _{\epsilon }^{-1}||_{s}\leq \frac{c_{s}}{\epsilon ^{2}}.
\end{equation*}
\end{lemma}
\begin{proof}
Let us assume that $f\in \mathcal{H}_{s},$ and define $f_{s}\in \mathcal{H}%
_{0}$ by its Fourier coefficients%
\begin{equation*}
f_{s}^{(\mathbf{k})}=(1+N_{\mathbf{k}}^{2})^{s/2}f^{(\mathbf{k)}}.
\end{equation*}%
We then have $\|f_s\|_0=\|f\|_s$.
Then $\Lambda _{\epsilon }U=f$ leads to $(\Lambda _{\epsilon }U)_{s}\in \mathcal{H}_{0},\text{ \ }||(\Lambda
_{\epsilon }U)_{s}||=||f||_{s}.$

By definition%
\begin{equation*}
(\Lambda _{\epsilon }U)^{(\mathbf{k})}=(1-|\mathbf{k|}^{2})^{2}U^{(\mathbf{k}%
)}+\epsilon ^{2}\Sigma _{\mathbf{l}\in \sigma _{2}}a^{(\mathbf{k-l})}U^{(%
\mathbf{l})},
\end{equation*}%
where $\mathbf{k}\in \sigma _{2}.$ 
Now 
\begin{equation*}
(\Lambda _{\epsilon }U)_{s}^{(\mathbf{k})}-(\Lambda _{\epsilon }U_{s})^{(%
\mathbf{k})}=\epsilon ^{2}\Sigma _{\mathbf{l}\in \sigma _{2}}[(1+N_{\mathbf{k%
}}^{2})^{s/2}-(1+N_{\mathbf{l}}^{2})^{s/2}]a^{(\mathbf{k-l})}U^{(\mathbf{l}%
)},
\end{equation*}%
and since $|N_{\mathbf{k}}-N_{\mathbf{l}}|\leq 2$ from the form of $a$, we have from Lemma \ref{little lemma}
\begin{equation*}
|(1+N_{\mathbf{k}}^{2})^{s/2}-(1+N_{\mathbf{l}}^{2})^{s/2}|\leq d(s/2,2)(1+N_{%
\mathbf{k}}^{2})^{s/2-1}.
\end{equation*}%
Now define $\widetilde{U}$ for any $U\in \mathcal{H}_s$ by
\begin{equation*}
\widetilde{U}^{(\mathbf {k})}=|\widetilde{U}^{(\mathbf {k})}|.
\end{equation*}
Then $||\widetilde{U}||_s=||U||_s$ and since for $0<s \leq 2$, $(1+N_{%
\mathbf{k}}^{2})^{s/2-1} \leq 1$,
\begin{equation*}
|[(\Lambda _{\epsilon }U)_{s}-(\Lambda _{\epsilon }U_{s})]^{(\mathbf{k})}|\leq\epsilon^2d(s/2,2))
(\widetilde a\widetilde U)^{(\mathbf{k})}
\end{equation*}
(where $\tilde{a}$ differs from the one defined at Lemma \ref{LemFirstDecomp}),
hence for $0<s \leq 2$
\begin{equation*}
||(\Lambda _{\epsilon }U)_{s}-(\Lambda _{\epsilon }U_{s})||_{0}\leq
d(s/2,2)\epsilon ^{2}||\widetilde{a}\widetilde{U}||_{0}\leq d_{s}\epsilon ^{2}||U||_{0}.
\end{equation*}
Hence we obtain%
\begin{equation*}
||\Lambda _{\epsilon }U_{s}||_{0}\leq ||(\Lambda _{\epsilon
}U)_{s}||_{0}+d_{s}/2||f||_{0}=||f||_{s}+d_{s}/2||f||_{0},
\end{equation*}%
and finally, for $0 \leq s \leq 2$
\begin{equation*}
||U||_{s}=||U_{s}||_{0}\leq \frac{1}{2\epsilon ^{2}}||\Lambda _{\epsilon
}U_{s}||_{0}\leq \frac{c_{s}}{\epsilon ^{2}}||f||_{s}.
\end{equation*}%
 Let us prove by induction on $s \geq 0$ that $||\Lambda_{\epsilon}^{-1}||_s\leq c_{s}\epsilon^{-2}$.
This holds for $0 \leq s\leq 2$. Assume that it holds for $s-2$, then
\begin{equation*}
|[(\Lambda _{\epsilon }U)_{s}-(\Lambda _{\epsilon }U_{s})]^{(\mathbf{k})}|\leq\epsilon^2d(s/2,2))
(1+N_{\mathbf{k}}^{2})^{s/2-1}(\widetilde a\widetilde U)^{(\mathbf{k})}.
\end{equation*}
Hence, we have
\begin{equation*}
||(\Lambda _{\epsilon }U)_{s}-(\Lambda _{\epsilon }U_{s})||_{0}\leq
d(s/2,2)\epsilon ^{2}||\widetilde{a}\widetilde{U}||_{s-2}\leq d_{s}\epsilon ^{2}||U||_{s-2}.
\end{equation*}%
We assumed that $||\Lambda _{\epsilon
}^{-1}||_{s-2}\leq \frac{c_{s-2}}{\epsilon ^{2}},$ hence we obtain%
\begin{equation*}
||\Lambda _{\epsilon }U_{s}||_{0}\leq ||(\Lambda _{\epsilon
}U)_{s}||_{0}+d_{s}c_{s-2}||f||_{s-2}=||f||_{s}+d_{s}c_{s-2}||f||_{s-2},
\end{equation*}%
hence%
\begin{equation*}
||U||_{s}=||U_{s}||_{0}\leq \frac{1}{2\epsilon ^{2}}||\Lambda _{\epsilon
}U_{s}||_{0}\leq \frac{c_{s}}{\epsilon ^{2}}||f||_{s}.
\end{equation*}%
This ends the proof of Lemma \ref{estimInvH_s}.
\end{proof}

Then we finally have

\begin{lemma}
\label{invL_eps} For any $q\geq 4,$ and  $s\geq 0,$there exists $\epsilon
_{0}>0,$ such that for $0<\epsilon \leq \epsilon _{0}$ the linear operator $%
\mathcal{L}_{\epsilon }$ has a bounded inverse in $\mathcal{H}_{s},$ with%
\begin{equation*}
||\mathcal{L}_{\epsilon }^{-1}||_{s}\leq \frac{c(s)}{\epsilon ^{2}},
\end{equation*}%
where $c(s)$ is a positive constant only depending on $s.$

\begin{proof}
From Lemma \ref{estimInvH_s} we have%
\begin{equation*}
(\mathcal{L}_{\epsilon }^{(2)})^{-1}=(1+\epsilon ^{4}\Lambda _{\epsilon
}^{-1}R_{\epsilon })^{-1}\Lambda _{\epsilon }^{-1},
\end{equation*}%
and $\TEXTsymbol{\vert}\TEXTsymbol{\vert}\epsilon ^{4}\Lambda _{\epsilon
}^{-1}R_{\epsilon }||_{s}\leq c_{s}\epsilon ^{2}||R_{\epsilon }||_{s}\leq
c_{s}^{\prime }\epsilon ^{2}.$ For $\epsilon $ small enough we then have%
\begin{equation*}
||(\mathcal{L}_{\epsilon }^{(2)})^{-1}||\leq \frac{2c_{s}}{\epsilon ^{2}}.
\end{equation*}%
Then, from Lemma \ref{LemFirstDecomp} we deduce immediately that there
exists a constant $c(s)$ such that%
\begin{equation*}
||U||_{s}\leq \frac{c(s)}{\epsilon ^{2}}||f||_{s},
\end{equation*}%
which proves the Lemma.
\end{proof}
\end{lemma}

\section{Existence of the solution}

Below we prove our main result
\begin{theorem}
For any $q\geq 4$ and for any $s>q/2,$ there exists $\lambda _{0}>0$ such
that for $0<\lambda <\lambda _{0},$ there exists a quasipattern solution of
the Swift-Hohenberg steady equation (\ref{eq:sh}) in $\mathcal{H}_{s}$,
bifurcating from 0 and invariant under rotations of angle $\pi /q.$ Its
asymptotic expansion at the origin is given by the formal expansion computed
in \cite{iooss-rucklidge}.
\end{theorem}

\begin{proof}
We want to solve (\ref{basic-equW}) with respect to $W$ in $\mathcal{H}%
_{s},s>q/2.$ Taking into account of Lemma \ref{invL_eps}, this equation
takes the following form for $\epsilon \neq 0:$%
\begin{equation}
W=-\epsilon ^{3}\mathcal{L}_{\epsilon }^{-1} [f_{\epsilon}
+3\epsilon U_{\epsilon }W^{2}+\epsilon ^{5}W^{3}],  \label{equW}
\end{equation}%
which we write as%
\begin{equation*}
W=\mathcal{G}(\epsilon ,W)
\end{equation*}%
where $\mathcal{G}$ is well defined in $]0,\ep_0[\times\mathcal{H}_{s},$ depending smoothly
on its arguments for $0<\epsilon <\epsilon _{0}$ and $W\in \mathcal{H}_{s}.$
In fact for $\epsilon \neq 0$ fixed, $\mathcal{G}$ is analytic in $W,$ and
observing that $\epsilon ^{3}\mathcal{L}_{\epsilon }^{-1}=O(\epsilon ),$ we
see that $\mathcal{G}$ is continuous in $\epsilon $ on $]0,\epsilon _{0}[.$

The map $\mathcal{G}$ is Lipschitz in $W$ in a fixed ball of $\mathcal{H%
}_{s},$ with a small Lipschitz constant for $%
\epsilon \in ]0,\epsilon _{0}[$. Indeed, we have
\begin{equation*}
\|\mathcal{G}(\epsilon ,W)-\mathcal{G}(\epsilon ,W')\|_s\leq \|\epsilon ^{3}\mathcal{L}_{\epsilon }^{-1}[3\epsilon U_{\epsilon}(W^{2}-W'^{2})+\epsilon ^{5}(W^{3}-W'^{3})]\|_s,
\end{equation*}
$\epsilon ^{3}\|\mathcal{L}_{\epsilon }^{-1}\|_s\leq c_s\epsilon $ and $\|U_{\ep}\|_s\leq c_s'\ep$. Moreover, we have
$$
\mathcal{G}(\epsilon ,0)=-\epsilon ^{3}\mathcal{L}_{\epsilon }^{-1}(f_{\epsilon})=0(\ep).
$$

Than according Dieudonn\'e's version of the implicit function theorem \cite{dieudonne1}[(10.1.1)], there exists a unique mapping $W(\ep)$ into a ball in ${\cal H}_s$ of size $O(\epsilon)$, such that 
$$
W(\ep)=\mathcal{G}(\epsilon ,W(\ep))
$$
for all $\ep\in]0,\ep_0[$ and $W$ is continuous there.
Finally we have a solution $U=U_{\epsilon }+\epsilon^4 W(\epsilon )$ of (\ref%
{eq:sh}) of the form (\ref{eq:perturbed}).
\end{proof}
\appendix
\section{Proof of Lemma \ref{algebrab}}
We follow and modify the argument of Iooss-Rucklidge \cite{iooss-rucklidge}[appendix C] used to prove that $\cH_s$ is an algebra \cite{iooss-rucklidge}[lemma 4.2]. 
Let 
$$
u=\sum_{\bk\in \Ga}u^{(\bk)}(\nu)e^{i\bk.\bf{x}},\quad v=\sum_{\bk\in \Ga}v^{(\bk)}(\nu)e^{i\bk.\bf{x}}
$$
be elements of $\cH_{s}\cap\cH_{s'}$.
We have 
$$
2^{-2s+1}\|uv\|_{s}^2\leq \underbrace{\sum_{\bK}\left\|\sum_{\bk+\bk'=\bK}u^{(\bk)}v^{(\bk')}\right\|_{p,K}^2(1+N_{\bk}^2)^s}_{S_1}+\underbrace{\sum_{\bK}\left\|\sum_{\bk+\bk'=\bK}u^{(\bk)}v^{(\bk')}\right\|_{p,K}^2(1+N_{\bk'}^2)^s}_{S_2}.
$$
Moreover, we have
$$
\frac{1}{2}S_1\leq \underbrace{\sum_{\bK}\left|\sum_{\substack{\bk+\bk'=\bK\\ N_{\bk}\leq 3 N_{\bk'}}}u^{(\bk)}v^{(\bk')}\right|^2(1+N_{\bk}^2)^s}_{S'_1}+\underbrace{\sum_{\bK}\left|\sum_{\substack{\bk+\bk'=\bK\\ N_{\bk} >3 N_{\bk'}}}u^{(\bk)}v^{(\bk')}\right|^2(1+N_{\bk}^2)^s}_{S''_1}.
$$
Using the fact that $1\leq \left(\frac{16(1+N_{\bk'}^{2})}{(1+N_{\bK}^{2})}\right)^{s'}$ and Cauchy-Schwarz inequality,  we obtain ($s'>q/2$)
$$
S'_1\leq (K^*)^2\|u\|_{s}^2\|v\|_{s'}^2\sum_{\bK}\left(\frac{16}{(1+N_{\bK}^{2})}\right)^{s'}\leq C\|u\|_s^2\|v\|_{s'}^2.
$$
To obtain a similar bound for $S''_1$, we use the Iooss-Rucklidge dyadic decomposition of 
$S''_1$~: $\De_pu:=\sum_{2^p\leq N_{\bk}<2^{p+1}}u^{(\bk)}e^{i\bk.\bf{x}}$, $\De_{-1}u:=u^{(\mathbf{0})}$ and $S_k u:=\sum_{p=-1}^k\De_pu$. Then, $u=\sum_{p\geq -1}\De_pu\in\cH_{s}$ if and only if $\sum_{p\geq -1}2^{2ps}\|\De_pu\|_{0}^2<+\infty$. According to the computation of \cite{iooss-rucklidge}[p. 387], we have
\beq\label{prodpartiel}
S''_1=\left\|\sum_{\bK}\left(\sum_{\substack{\bk+\bk'=\bK\\ N_{\bk} >3 N_{\bk'}}}u^{(\bk)}v^{(\bk')}\right)e^{i\bK. \bf{x}}\right\|_{s}\leq C\sum_{j=-1}^{+\infty}2^{2js}\left\|\left(\sum_{p=j-1}^{j+1}\De_j(S_{p-1}v\De_pu)\right)\right\|_{0}^2
\eeq
Since $s'>q/2$, by Cauchy-Schwarz, we have $\sum |v^{(\bk')}| \leq c\|v\|_{s'}$. So, following the computations \cite{iooss-rucklidge}[p.388], we obtain $\|S_{p-1}v\De_pu\|_{0}^{2}\leq C\|\De_pu\|_{0}^2\|v\|_{s'}^2$. From this and following the same computation as in \cite{iooss-rucklidge}[p. 388], we obtain
$$
S''_1\leq C''\|v\|_{s'}^2\sum_{p=-1}^{\infty}2^{2ps}\|\De_pu\|_{0}^2\leq C_1\|v\|_{s'}^2\|u\|_{s}^2
$$
To get an estimate for $S_2$, we just need to interchange the role of $u$ and $v$ and the result is proved.

\section{Proof of Lemma \ref{little lemma}}
We assume $x,y>0,$ and $p>0$ and%
\[
|x-y|\leq K.
\]%
Then, we prove that 
\[
|(1+x)^{p}-(1+y)^{p}|\leq d(p)(1+x)^{p-1},
\]%
with%
\begin{eqnarray*}
d(p) &=&pK(1+K)\text{ for }p>1, \\
&=&pK(1+K)^{1-p} \text{ for }p\leq 1.
\end{eqnarray*}%
\begin{proof}
For some $t$ between $x$ and $y$, we have
\begin{equation*}
(1+x)^p-(1+y)^p=p(x-y)(1+t)^{p-1}.
\end{equation*}

If $p \geq 1$ we use $(1+t)/(1+x) =1+(t-x)/(1+x) \leq 1+K$  if  $x <t < y$ and
$\leq 1$ if $y<t <x $. This proves the lemma if $p \geq 1$.

Similarly if $p \leq 1$ we use $(1+x)/(1+t) \leq 1 $ if  $x<t<y$  and 
$=1+(x-t)/(1+t)\leq
1+K$  if  $y <t<x$ and so $(1+t)^{p-1}\leq const.(1+x)^{p-1}$.\end{proof}

\end{document}